\documentclass[a4paper,11pt,oneside]{article}
\usepackage{graphicx}
\usepackage{amscd}
\usepackage{amsmath}
\usepackage{caption}
\usepackage{amsfonts}
\usepackage{amssymb}
\usepackage{amsthm}
\usepackage{mathrsfs}
\usepackage{multicol}
\usepackage{color}
\usepackage[english]{babel}
\usepackage[T1]{fontenc}
\usepackage{textcomp}
\usepackage[utf8]{inputenc}
\usepackage{enumitem}
\usepackage{indentfirst}

\newcommand{\N}{\mathbb N}

\newcommand{\RR}{{{\rm I} \kern -.15em {\rm R} }}

\begin{document}
	\theoremstyle{plain} \newtheorem{thm}{Theorem}[section] \newtheorem{cor}[thm]{Corollary} \newtheorem{lem}[thm]{Lemma} \newtheorem{prop}[thm]{Proposition} \theoremstyle{definition} \newtheorem{defn}{Definition}[section] 
	
	\newtheorem{oss}[thm]{Remark}
	\newtheorem{ex}{Example}[section]
	\newtheorem{lemma}{Lemma}[section]
	\title{First and second-order Cucker-Smale models with non-universal interaction, time delay and communication failures}
	\author{Chiara Cicolani\footnote{ Email: chiara.cicolani@graduate.univaq.it. },\hspace{0.2cm}Elisa Continelli\footnote{Email: elisa.continelli@graduate.univaq.it.}\hspace{0.3cm}\&\hspace{0.2cm}Cristina Pignotti\footnote{Email: cristina.pignotti@univaq.it.} \\Dipartimento di Ingegneria e Scienze dell'Informazione e Matematica\\
		Universit\`{a} degli Studi di L'Aquila\\
		Via Vetoio, Loc. Coppito, 67100 L'Aquila Italy}

	\maketitle

	\begin{abstract}
		In this paper, we deal with first and second-order alignment models with non-universal interaction, time delay and possible lack of connection between the agents. More precisely, we analyze the situation in which the system's agents do not transmit information to all the other agents  and also agents that are linked to each other can suspend their interaction at certain times. Moreover, we take into account of possible time lags in the interactions. To deal with the considered "non-universal" connection, a graph topology over the structure of the model has to be considered. Under a so-called Persistence Excitation Condition, we establish the exponential convergence to consensus for both models whenever the digraph that describes the interaction between the agents is strongly connected.
	\end{abstract}

	\providecommand{\keywords}[1]{\textbf{Keywords:} #1}
	
	\vspace{5 mm}
	
	\section{Introduction}
	Multiagent systems have been deeply investigated in these last years, due to their wide application to several scientific disciplines, among others biology \cite{Cama, CS1}, economics \cite{Marsan}, robotics \cite{Bullo}, control theory \cite{Borzi, PRT, WCB}, social sciences \cite{Bellomo, CF, Campi}. Among them, there is the Hegselmann-Krause opinion formation model \cite{HK} and its second-order version, the Cucker-Smale model, introduced in \cite{CS1} for the description of flocking phenomena, such as the flocking of birds, the swarming of bacteria and the schooling of fish (see also \cite{Ha1, Ha2}). Also, let us mention the Kuramoto model \cite{Kuramoto}, that describes the behavior of a large set of coupled oscillators.  Typically, for the solutions of the aforementioned multiagent systems, the approach to consensus, in the case of the Hegselmann-Krause model, the exhibition of asymptotic flocking, in the case of the Cucker-Smale  model, and the asymptotic synchronization, in the case of the case of the Kuramoto model, are investigated.
	
	In multiagent systems, it is important  to consider the presence  of time delay effects. Indeed, in the applications, one has to take into account of certain time lags due to propagation of information among the agents or to reaction times. 
	
	The analysis of the Hegselmann-Krause model and of the Cucker-Smale model in presence of time delays (that can be constant or, more realistically, varying in time), has been carried out by many authors, \cite{CH, CL, CPP, CP, DH, H, H3, HM, LW, Lu,  P, PT}. Most of them require an upper bound on the time delay size to get the convergence to consensus. However, Rodriguez Cartabia proved in \cite{Cartabia} the asymptotic flocking for the Cucker-Smale model with constant time delay without requiring any smallness condition on the time delay size (see also \cite{H3} for a consensus result for the Hegselmann-Krause model). Generalizing and extending the arguments employed in \cite{Cartabia}, the exponential convergence to consensus for the Hegselmann-Krause model and the exhibition of asymptotic flocking for the Cucker-Smale model with time-variable time delay have been proved in \cite{ContPign} and \cite{Cont}, respectively, without assuming the time delay size to be small. Also, in \cite{Cont, ContPign} monotonicity assumptions on the influence function, which are usually required when dealing with the Hegselmann-Krause model and with the Cucker-Smale model, are removed, namely the influence function is assumed to be just positive bounded and continuous. In the same spirit, the asymptotic synchronization of the Kuramoto oscillators with time delay and non-universal interaction has been achieved in \cite{ChoiCicoPi}.
	
	A possible scenario that could occur in the analysis of such models is the one in which the system's particles suspend the interactions they have with the other agents. As a consequence, there is a lack of connection between the system's elements that, of course, obstacles the convergence to consensus, for the first-order model, or the flocking for the second-order one. Then, it is important to find conditions guaranteeing the system's alignment.
	
	In \cite{Bonnet}, the convergence to consensus and the asymptotic flocking for a class of Cucker-Smale systems under communication failures among the system's agents have been proved under suitable assumptions in the case of symmetric interaction coefficients. Also, in \cite{AnconaRossi}, the approach to consensus for a first-order alignment system involving weights depending on the couple of agents that can eventually degenerate has been proved under a Persistence Excitation Condition. In the case of nonsymmetric interaction coefficients, the exponential convergence to consensus for the Hegselmann-Krause model with time delay and lack of interaction has been obtained in \cite{ContPi}.
	
	Also, it could happen that the agents involved in an opinion formation or flocking process are not able to exchange information with all the other components of the system. In this case, we are in presence of a non-universal interaction, so that the agents are able to influence only the opinions or the velocities of the particles they are linked to. To deal with this kind of interaction, a graph topology over the structure of the model has to be considered (see \cite{ChoiCicoPi}). 
	
	In this paper, we analyze the asymptotic behavior of the solutions to first and second-order alignment models,  i.e. Hegselmann-Krause and Cucker-Smale models,  in presence of (pair and time-dependent) time delays, non-universal interaction, and lack of connection. Namely, in these models, weight functions depending on the agents' pair and that can degenerate at certain times  are considered. In this case, the interaction is lacking not only among agents that are not linked each other but also among agents that are generally able to exchange information. Under a Persistence Excitation Condition, we establish the exponential consensus and flocking for the Hegselmann-Krause opinion formation model and for the Cucker-Smale model whenever the digraph that describes the interaction among the agents is strongly connected. 

This is done dealing with a very general influence function and without requiring any smallness asssumption on the time delay size. Our result is very general and greatly improves
previous related works \cite{Bonnet, ContPi, AnconaRossi}.  Indeed, in \cite{AnconaRossi}
only the first-order model is analyzed and  the convergence to consensus is not achieved exponentially fast. Furthermore, time delay effects are not considered. Here, we work in a more general setting since there are particles that could never communicate with each other, and we consider time delays, pair and time depending. This wide generality  requires finer and sophisticated arguments.
With respect to \cite{ContPi}, where only the first-order model is analyzed, the main novelties are the more general weight and time delay functions (now pair-dependent). Moreover, here, we work in the network topology setting. Finally, \cite{Bonnet} deals with the second-order model too. However, the analysis requires symmetry conditions on  the weight functions and all-to-all interaction. Moreover, no time delays are included.

	The present paper is so organized. In Section \ref{timedelay}, we establish the exponential convergence to consensus for the solutions of the Hegselmann-Krause model, namely the first-order alignment model.
In Section \ref{secdelay}, the discussion is extended to the second-order version, namely the Cucker-Smale model. The proof of the exponential flocking for the Cucker-Smale model requires a more careful analysis with respect to the one carried out in Section \ref{timedelay}, since in this case we have to prove that the agents uniform their velocities and the distances between the agents' positions are bounded. 
	\section{The first-order alignment model}\label{timedelay}
	Consider a finite set of $N\in\N$ agents, with $N\geq 2 $. Let $x_{i}(t)\in \RR^d$ be the opinion of the $i$-th agent at time $t$. We shall denote with $\lvert\cdot \rvert$ and $\langle\cdot,\cdot\rangle$ the usual norm and scalar product on $\RR^{d}$, respectively. Let us denote with $\N_0:=\N\cup\{0\}.$ The interactions between the elements of the system are described by the following Hegselmann-Krause type model:
	\begin{equation}\label{onoff}
		\frac{d}{dt}x_{i}(t)=\underset{j:j\neq i}{\sum}\chi_{ij}\alpha_{ij}(t) b_{ij}(t)(x_{j}(t-\tau_{ij}(t))-x_{i}(t)),\quad t>0,	\,\, \forall i=1,\dots,N.
	\end{equation}
		where the time delay functions $\tau_{ij}:[0,+\infty)\rightarrow[0,+\infty)$ are assumed to be continuous and satisfy the following:
	\begin{equation}\label{taubounded}
		0\leq \tau_{ij}(t)\leq \tau,\quad \forall t\geq 0, \,\forall i,j=1,\dots,N,
	\end{equation}
	for a suitable positive constant $\tau$. \\Here, the communication rates $b_{ij}$ are of the form
	\begin{equation}\label{weight}
		b_{ij}(t):=\frac{1}{N-1}\psi( x_{i}(t), x_{j}(t-\tau_{ij}(t))), \quad\forall t>0,\, \forall i,j=1,\dots,N,
	\end{equation}
	where the influence function $\psi:\RR^d\times\RR^d\rightarrow \RR$ is positive, bounded and continuous and
	\begin{equation}\label{K}
		K:=\lVert \psi\rVert_{\infty}.
	\end{equation} 
	The terms $\chi_{ij}$ are so defined
	\begin{equation}\label{chiij}
		\begin{cases}
			1,\quad \text{if } j \text{ transmits information to } i,\\
			0,\quad \text{otherwise}.
		\end{cases}
	\end{equation}
The weight functions $\alpha_{ij}:[0,+\infty)\rightarrow [0,1]$ are $\mathcal{L}^1$-measurable and satisfy the following Persistence Excitation Condition (cf. \cite{Bonnet, AnconaRossi}): 

\begin{itemize}
\item [\bf (PE)]
there exist two positive constants $T$ and $\tilde{\alpha}$ such that
\begin{equation}\label{PE}
	\int_{t}^{t+T}\alpha_{ij}(s)ds\geq \tilde{\alpha},\quad \forall t\geq 0,
\end{equation}
for all $i,j=1,\dots,N$ such that $\chi_{ij}=1$. Without loss of generality, we can assume that $\tilde{\alpha}K\leq 1$. 
\end{itemize}

Let us note that \eqref{PE} becomes relevant when $T$ is large and $\tilde{\alpha}$ is small. In this case, the agents could eventually suspend their interaction for long enough.  We also point out that, in the case in which $\alpha_{ij}(t)=1$, for a.e. $t\geq 0$ and for any $i,j=1,\dots,N$, i.e. in the case in which the agents do not interrupt their exchange of information, the condition \eqref{PE} is of course satisfied.
\\Due to the presence of the time delay, the initial conditions are functions defined in the interval $[-\tau, 0].$ The initial conditions
	\begin{equation}\label{incond}
		x_{i}(s)=x^{0}_{i}(s),\quad \forall s\in [-\tau,0],\,\forall i=1,\dots,N,
	\end{equation}
	are assumed to be continuous functions. 
	\\We set
	\begin{equation}\label{M0}
		C_{0}:=\max_{i=1,\dots,N}\,\,\max_{s\in [-\tau, 0]}\lvert x_{i}(s)\rvert,
	\end{equation} \begin{equation}\label{psi0}
		\psi_{0}:=\min_{\vert y\vert, \vert z\vert \le  C_{0}}\psi(y,z).
	\end{equation}
	We will consider a graph topology over the model structure. Let $ \mathcal{G}=(\mathcal{V}, \mathcal{E})$ be a digraph consisting of a finite set $\mathcal{V}=\{1,...,N \}$ of vertices and a set $\mathcal{E} \subset \mathcal{V} \times \mathcal{V}$ of arcs. We assume that opinions of the agents are located at the vertices and interact each other via the underlying network topology. For each vertex $i$, we denote by $\mathcal{N}_i$ the set of vertices that directly influence the vertex $i$, namely
	\begin{equation}\label{N_i}
		\mathcal{N}_i:=\{j= 1,\dots,N:\chi_{ij}=1\}.
	\end{equation} 
	The set $\mathcal{N}_i$ can also be defined in the following way: $j \in \mathcal{N}_i$ if and only if $(i,j)\in \mathcal{E}$. Also, we denote with
	\begin{equation}\label{cardN_i}
		N_i:=|\mathcal{N}_i|.
	\end{equation}
	Throughout the paper, we will exclude self loops, i.e. we assume that $i \notin \mathcal{N}_i$ for all $1 \leq i \leq N.$ We also denote the network topology via its $(0,1)$-adjacency matrix $(\chi_{ij})_{ij}$. A \emph{path} in a digraph $\mathcal{G}$ from $i_0$ to $i_p$ is a finite sequence $i_0, i_1, \dots, i_p$ of distinct vertices such that each successive pair of vertices is an arc of $\mathcal{G}.$ The integer p is called \emph{length} of the path. If there exists a path from $i$ to $j$, then vertex $j$ is said to be \emph{reachable} from vertex $i$ and we define the distance from $i$ to $j$, in notation $\text{dist}(i,j)$, as the length of the shortest path from $i$ to $j$. A digraph $\mathcal{G}$ is said to be \emph{strongly connected} if each vertex is reachable from any other vertex. We assume that our digraph $\mathcal{G}$ is strongly connected. We define the \emph{depth} $\gamma$ of the digraph as follows:
	\begin{equation}\label{gamma}
		\gamma:=\max_{i,j =1,\dots, N} \text{dist}(i,j). 
	\end{equation}
	Thus, any particle can be connected to the other individuals of the system via no more than $\gamma$ intermediate agents. By definition, since $i\notin\mathcal{N}_i$, for all $i=1,\dots,N$, we have that $\gamma \leq N-1.$ Also, since the digraph is strongly connected, the depth $\gamma\geq 1$.

	Now, we give the rigorous definition of convergence to consensus for solutions of the Hegselmann-Krause model \eqref{onoff}. We define the diameter $d(\cdot)$ of the solution as
	$$d(t):=\max_{i,j=1,\dots,N}\lvert x_{i}(t)-x_{j}(t)\rvert,\quad \forall t\geq-\tau.$$
	
	\begin{defn}\label{consensus}
		We say that a solution $\{x_{i}\}_{i=1,\dots,N}$ to system \eqref{onoff} converges to \textit{consensus} if $d(t)\to0$, as $t\to \infty$.
	\end{defn}
	We will prove the following exponential convergence to consensus result.
	\begin{thm}\label{consgenonoff}
		Assume \eqref{taubounded} and that the digraph $\mathcal{G}$ is strongly connected. Let $\psi:\RR^d\times \RR^d\rightarrow\RR$ be a positive, bounded, continuous function. Assume that the weight functions $\alpha_{ij}:[0,+\infty)\rightarrow[0,1]$ are ${\mathcal L^1}$-measurable and satisfy {\bf (PE)}. Let $x_i^0:[-\tau,0]\rightarrow\RR^d$ be a continuous function, for any $i=1,\dots,N$. Then, every solution $\{x_{i}\}_{i=1,\dots,N}$ to \eqref{onoff} with the initial conditions \eqref{incond} satisfies the following exponential decay estimate
		\begin{equation}\label{expconsonoff}
			d(t)\leq\left( \max_{i,j =1,\dots, N}\max_{r,s\in[-\tau,0]}\lvert x_i(r)-x_j(s)\rvert \right)e^{- C(t-\gamma(T+\tau)-\tau)},\quad \forall t\geq 0,
		\end{equation}
		where $\gamma>0$ is the depth of the digraph, $T$ is the positive constant in \eqref{PE} and $C$ is a suitable positive constant.
	\end{thm}
	
	\subsection{Preliminary lemmas}
	Let $\{x_{i}\}_{i=1,\dots,N}$ be solution to \eqref{onoff} under the initial conditions \eqref{incond}. We assume that the hypotheses of Theorem \ref{consgenonoff} are satisfied. We present some auxiliary lemmas.
		\begin{defn}\label{quantonoff}
		Given a vector $v\in \mathbb{R}^d$, for all $n\in \mathbb{N}_0$ we define
		$$I_n:=[n(\gamma (T+\tau)+\tau)-\tau,n(\gamma (T+\tau)+\tau)]$$
		$$m_n^v:=\min_{i=1,\dots,N}\min_{s\in I_n}\,\langle x_{i}(s),v\rangle,$$
		$$M_n^v:=\max_{j=1,\dots,N}\max_{s\in I_n}\,\langle x_{j}(s),v\rangle.$$
		Also, we define, for all $n\in\mathbb{N}_0$,
		$$\tilde{m}_n^v:=\min_{i=1,\dots,N}\langle x_{i}(n(\gamma (T+\tau)+\tau)),v\rangle,$$
		$$\tilde{M}_n^v:=\max_{j=1,\dots,N}\langle x_{j}(n(\gamma (T+\tau)+\tau)),v\rangle.$$
	\end{defn} 
	\begin{lem}\label{L1}
		For each vector $v\in \RR^{d}$, we have that 
		\begin{equation}\label{scalpronoff}
			m_0^v\leq \langle x_{i}(t),v\rangle \leq M_0^v,
		\end{equation}for all  $t\geq -\tau$ and for any $i=1,\dots,N$.
	\end{lem}
\begin{proof}
	First of all, we note that the inequalities in \eqref{scalpronoff} are satisfied for every  $t\in [-\tau,0]$.
	\\Now, let $v\in \RR^{d}$. For all $\epsilon >0$, we define
	$$K^{\epsilon}:=\left\{t>0 :\max_{i=1,\dots,N}\langle x_{i}(s),v\rangle < M_0^v+\epsilon,\,\forall s\in [0,t)\right\},$$
	and $$S^{\epsilon}:=\sup K^{\epsilon}.$$
	By continuity, we have that $K^{\epsilon}\neq\emptyset$ and $S^{\epsilon}>0$. 
	\\We claim that $S^{\epsilon}=+\infty$. Indeed, suppose by contradiction that $S^{\epsilon}<+\infty$. By definition of $S^{\epsilon}$, it turns out that \begin{equation}\label{max}
		\max_{i=1,\dots,N}\langle x_{i}(t),v\rangle<M_0^v+\epsilon,\quad \forall t\in (0,S^{\epsilon}),
	\end{equation}
	\begin{equation}\label{teps}
		\lim_{t\to S^{\epsilon-}}\max_{i=1,\dots,N}\langle x_{i}(t),v\rangle=M_0^v+\epsilon.
	\end{equation}
	For all $i=1,\dots,N$ and $t\in (0,S^{\epsilon})$, we have that
	$$\frac{d}{dt}\langle x_{i}(t),v\rangle=\frac{1}{N-1}\sum_{j:j\neq i}\chi_{ij}\alpha_{ij}(t)\psi(x_{i}(t), x_{j}(t-\tau_{ij}(t)))\langle x_{j}(t-\tau_{ij}(t))-x_{i}(t),v\rangle.$$
	Now, being $t\in (0,S^{\epsilon})$, it holds that $t-\tau_{ij}(t)\in (-\tau, S^{\epsilon})$. Then, from \eqref{max}
	\begin{equation}\label{t-tau}
		\langle x_{j}(t-\tau_{ij}(t)),v\rangle< M_0^v+\epsilon,\quad \forall j=1, \dots, N,
	\end{equation}
	where hare we have used the fact that the second inequality in \eqref{scalpronoff} is satisfied in $[-\tau,0]$.
	\\Therefore, using \eqref{K}, \eqref{max}, \eqref{t-tau} and recalling that $\chi_{ij},\alpha_{ij}\leq 1$, for a.e. $t\in (0, S^{\epsilon})$ we can write $$\frac{d}{dt}\langle x_{i}(t),v\rangle\leq \frac{1}{N-1}\sum_{j:j\neq i}\chi_{ij}\alpha_{ij}(t)\psi(x_{i}(t), x_{j}(t-\tau_{ij}(t)))(M_0^v+\epsilon-\langle x_{i}(t),v\rangle)$$$$\leq K(M_0^v+\epsilon-\langle x_{i}(t),v\rangle).$$
	Thus, the Gronwall's inequality yields
	$$\begin{array}{l}
		\vspace{0.2cm}\displaystyle{
			\langle x_{i}(t),v\rangle\leq e^{-Kt}\langle x_{i}(0),v\rangle+K(M_0^v+\epsilon)\int_{0}^{t}e^{-K(t-s)}ds}\\
		\vspace{0.3cm}\displaystyle{\hspace{1.7 cm}
			=e^{-Kt}\langle x_{i}(0),v\rangle+(M_0^v+\epsilon)e^{-Kt}(e^{Kt}-1)}\\
		\vspace{0.3cm}\displaystyle{\hspace{1.7 cm}
			=e^{-Kt}\langle x_{i}(0),v\rangle+(M_0^v+\epsilon)(1-e^{-Kt})}\\
		\vspace{0.3cm}\displaystyle{\hspace{1.7 cm}
			\leq e^{-Kt}M_0^v+M_0^v+\epsilon -M_0^ve^{-Kt}-\epsilon e^{-Kt}}\\
		\vspace{0.3cm}\displaystyle{\hspace{1.7 cm}
			=M_0^v+\epsilon-\epsilon e^{-Kt}}\\
		\displaystyle{\hspace{1.7 cm}
			\leq M_0^v+\epsilon-\epsilon e^{-KS^{\epsilon}},}
	\end{array}
	$$
	for all $t\in (0, S^{\epsilon})$.	We have so proved that, $\forall i=1,\dots, N,$
	$$\langle x_{i}(t),v\rangle\leq M_0^v+\epsilon-\epsilon e^{-KS^{\epsilon}}, \quad \forall t\in (0,S^{\epsilon}).$$
	Thus, we get
	\begin{equation}\label{lim}
		\max_{i=1,\dots,N} \langle x_{i}(t),v\rangle\leq M_0^v+\epsilon-\epsilon e^{-KS^{\epsilon}}, \quad \forall t\in (0,S^{\epsilon}).
	\end{equation}
	Letting $t\to S^{\epsilon-}$ in \eqref{lim}, from \eqref{teps} we have that $$M_0^v+\epsilon\leq M_0^v+\epsilon-\epsilon e^{-KS^{\epsilon}}<M_0^v+\epsilon,$$
	which is a contradiction. Thus, $S^{\epsilon}=+\infty$ and $$\max_{i=1,\dots,N}\langle x_{i}(t),v\rangle<M_0^v+\epsilon, \quad \forall t>0.$$
	From the arbitrariness of $\epsilon$ we can conclude that $$\max_{i=1,\dots,N}\langle x_{i}(t),v\rangle\leq M_0^v, \quad \forall t>0,$$
	from which $$\langle x_{i}(t),v\rangle\leq M_0^v, \quad \forall t>0, \,\forall i=1,\dots,N.$$
	So, the second inequality in \eqref{scalpronoff} is proven. 
	\\Now, to show that the other inequality holds, fix $v\in \RR^{d}$. Then, for all $i=1,\dots,N$ and $t>0$, by applying the second inequality in \eqref{scalpronoff} to the vector $-v\in\RR^{d}$ we get $$-\langle x_{i}(t),v\rangle=\langle x_{i}(t),-v\rangle\leq \max_{j=1,\dots,N}\max_{s\in[-\tau,0]}\langle x_{j}(s),-v\rangle$$$$=-\min_{j=1,\dots,N}\min_{s\in [-\tau,0]}\langle x_{j}(s),v\rangle=-m_0^v,$$
	from which $$\langle x_{i}(t),v\rangle\geq m_0^v,\quad \forall t\geq 0,\,\forall i=1,\dots,N.$$
	Thus, also the first inequality in \eqref{scalpronoff} is fulfilled.
\end{proof}
Using the same arguments employed in the proof of the previous lemma, one can prove the following more general result.
	\begin{lem}\label{L1onoff}
	For each vector $v\in \RR^{d}$ and for all $n\in \mathbb{N}_0$, we have that 
	\begin{equation}\label{scalprn}
		m_n^v\leq \langle x_{i}(t),v\rangle \leq M_n^v,
	\end{equation}for all  $t\geq n(\gamma (T+\tau)+\tau)-\tau$ and for any $i=1,\dots,N$.
\end{lem}
Now, we define the following quantities.
\begin{defn}\label{Dn}
For all $n\in \mathbb{N}$, we define
	$$D_n:=\max_{i,j =1,\dots, N}\max_{r,s\in I_n}\lvert x_i(r)-x_j(s)\rvert.$$
\end{defn}
Let us note that, for $n=0$, 
$$D_0:=\max_{i,j =1,\dots, N}\max_{r,s\in I_0}\lvert x_i(r)-x_j(s)\rvert=\max_{i,j =1,\dots, N}\max_{r,s\in [-\tau,0]}\lvert x_i(r)-x_j(s)\rvert.$$
So, the exponential decay estimate in \eqref{expconsonoff} can be written as
$$d(t)\leq e^{-C(t-\gamma(T+\tau)-\tau)}D_0,\quad \forall t\geq 0.$$
	\begin{lem}\label{lemmadiamonoff}
		For each $n\in \mathbb{N}_0$, we have that 
		\begin{equation}\label{diamonoff}
			\lvert x_i(s)-x_j(t)\rvert\leq D_n,
		\end{equation}
		for all  $s,t\geq n(\gamma (T+\tau)+\tau)-\tau$ and for any $i,j=1,\dots,N$.
	\end{lem}

	\begin{proof}
	Fix $n\in \mathbb{N}_0$. Let $i,j=1,\dots,N$ and $s,t\geq n(\gamma (T+\tau)+\tau)-\tau$. Then, if $\lvert x_i(s)-x_j(t)\rvert=0$, \eqref{diamonoff} is obviously satisfied. So we can assume $\lvert x_i(s)-x_j(t)\rvert>0$. Let us define the unit vector 
	$$v=\frac{x_i(s)-x_j(t)}{\lvert x_i(s)-x_j(t)\rvert}.$$
	Then, using \eqref{scalprn} and Cauchy-Schwarz inequality, we have that 
	$$\lvert x_i(s)-x_j(t)\rvert=\langle x_i(s)-x_j(t),v\rangle=\langle x_i(s),v\rangle-\langle x_j(t),v\rangle\leq M_n^v-m_n^v$$$$\leq \max_{k,l=1,\dots,N}\max_{r,\sigma\in I_n}\lvert x_k(r)-x_l(\sigma)\rvert=D_n.$$
\end{proof}
\begin{oss}
	Note that \eqref{diamonoff} yields
	\begin{equation}\label{diamonoff2}
		d(t)\leq D_n,\quad \forall t\geq n(\gamma (T+\tau)+\tau)-\tau.
	\end{equation}
Moreover, from \eqref{diamonoff} it comes that
\begin{equation}
	D_{n+1}\leq D_n,\quad\forall n\in \mathbb{N}_0.
\end{equation}
\end{oss}
Next, we show that the agents' opinions are bounded by a constant that depends of the initial data.
	\begin{lem}\label{L3onoff}
		For every $i=1,\dots,N,$ we have that \begin{equation}\label{boundsolonoff}
			\lvert x_{i}(t)\rvert\leq C_{0},\quad \forall t\geq-\tau,
		\end{equation}
		where $C_{0}$ is the constant defined in \eqref{M0}.	
	\end{lem}
	\begin{proof}
	Given $i=1,\dots,N$ and $t\geq-\tau$, if $\lvert x_{i}(t)\rvert =0$, then trivially $C_{0}\geq \lvert x_{i}(t)\rvert $. On the contrary, if $\lvert x_{i}(t)\rvert >0$, we define $$v=\frac{x_{i}(t)}{\lvert x_{i}(t)\rvert},$$
	which is a unit vector.	Then, by applying \eqref{scalpronoff} and by using the Cauchy-Schwarz inequality, we get $$\lvert x_{i}(t)\rvert=\langle x_{i}(t),v\rangle\leq M_0^v=\max_{j=1,\dots,N}\max_{s\in [-\tau,0]}\langle x_{j}(s),v\rangle $$$$\leq\max_{j=1,\dots,N}\max_{s\in [-\tau,0]}\lvert x_{j}(s)\rvert\lvert v\rvert=\max_{j=1,\dots,N}\max_{s\in [-\tau,0]}\lvert x_{j}(s)\rvert=C_{0},$$
	and \eqref{boundsolonoff} is satisfied.
\end{proof}
	\begin{oss}\label{R1onoff}
		From the estimate \eqref{boundsolonoff}, since the influence function $\psi$ is continuous, we deduce that 
		\begin{equation}\label{stima_psionoff}
			\psi (x_i(t), x_j(t-\tau_{ij}(t)))\ge \psi_{0},
		\end{equation}
		for all $t\ge 0,$ for all $i,j=1,\dots, N,$ where $\psi_{0}$ is the positive constant in \eqref{psi0}.
	\end{oss}
	
	\subsection{Consensus estimate}
In order to prove the consensus result, we need the following crucial proposition, inspired by a previous argument in \cite{H3}.
	 
	\begin{prop}\label{lemma 3onoff}
		For all $v\in \mathbb{R}^d$, it holds
		\begin{equation} \label{Bonoff}
			m_{0}^v+\Gamma(\tilde{M}^v_{0}-m^v_{0}) \leq \langle x_i(t),v\rangle\leq M^v_{0}-\Gamma(M^v_{0}-\tilde{m}_{0}^v), 
		\end{equation}
		for all $t\in I_1$ and for all $i =1,\dots, N$, where $\Gamma$ is the positive constant defined as follows
		\begin{equation}\label{Gammaonoff}
			\Gamma:=e^{-K (\frac{1}{2}(\gamma^2+3\gamma)(T+\tau)+\tau)}\left(\frac{\psi_{0}\tilde{\alpha}}{N-1}\right)^{\gamma}.
		\end{equation}
	\end{prop}
	\begin{oss}
		Let us note that, from {\bf (PE)}, $\Gamma\in (0,1)$ since $\tilde\alpha \psi_0\leq \tilde\alpha K\leq 1$.
	\end{oss}
	\begin{proof}
		Fix $v \in \mathbb{R}^d$. Let $L=1,\dots,N$ be such that $\langle x_L(0),v\rangle= \tilde{m}_{0}^v $. Note that from \eqref{scalpronoff}, $M^v_0\geq \tilde{m}^v_0$. Then, for a.e. $t\in [0,\gamma(T+\tau)+\tau]$, using \eqref{scalpronoff} we have
		$$\begin{array}{l}
			\vspace{0.3cm}\displaystyle{\frac{d}{dt}\langle x_L(t),v\rangle=\sum_{j:j\neq L}\chi_{Lj}\alpha_{Lj}(t)b_{Lj}(t)(\langle x_j(t-\tau_{Lj}(t)),v\rangle-\langle x_L(t),v\rangle)}\\
			\vspace{0.3cm}\displaystyle{\hspace{1.5cm}\leq \sum_{j:j\neq L}\chi_{Lj}\alpha_{Lj}(t)b_{Lj}(t)(M_0^v-\langle x_L(t),v\rangle)}\\
			\displaystyle{\hspace{1.5cm}\leq \frac{K}{N-1}\sum_{j:j\neq L}(M_0^v-\langle x_L(t),v\rangle)=K(M_0^v-\langle x_L(t),v\rangle).}
		\end{array}$$
		Thus, the Gronwall's inequality yields
		$$\begin{array}{l}
			\vspace{0.3cm}\displaystyle{\langle x_L(t),v\rangle\leq e^{-Kt}\langle x_L(0),v\rangle+M_0^v(1-e^{-Kt})}\\
			\vspace{0.3cm}\displaystyle{\hspace{1.5cm}=e^{-Kt}\tilde{m}_{0}^v+M_0^v(1-e^{-Kt})}\\
			\vspace{0.3cm}\displaystyle{\hspace{1.5cm}=M_0^v-e^{-Kt}(M_0^v-\tilde{m}_{0}^v)}\\
			\displaystyle{\hspace{1.5cm}\leq M_0^v-e^{-K(\gamma(T+\tau)+\tau)}(M_0^v-\tilde{m}_{0}^v).}
		\end{array}$$
	Hence,
		\begin{equation}\label{firstgron}
			\langle x_L(t),v\rangle\leq M_0^v-e^{-K(\gamma(T+\tau)+\tau)}(M_0^v-\tilde{m}_{0}^v),\quad \forall t\in [0,\gamma (T+\tau)+\tau].
		\end{equation}
		Now, let $i_1=1,\dots,N\setminus \{L\}$ be such that $\chi_{i_1L}=1$. Such an index $i_1$ exists since the digraph is strongly connected. Then, for a.e. $t \in [\tau,\gamma (T+\tau)+\tau]$, from \eqref{firstgron} we get
		$$\begin{array}{l}
			\vspace{0.3cm}\displaystyle{\frac{d}{d t}  \langle x_{i_1}(t),v\rangle =  \sum_{j \neq i_1, L} \chi_{i_1j}\alpha_{i_1j}(t) b_{i_1j}(t)( \langle x_{j}(t-\tau_{i_1j}(t)),v\rangle-\langle x_{i_1}(t),v\rangle) }\\
			\vspace{0.3cm}\displaystyle{\hspace{2cm}+ \alpha_{i_1L}(t)b_{i_1L}(t)(\langle x_{L}(t-\tau_{i_1L}(t)),v\rangle-\langle x_{i_1}(t),v\rangle)}\\
			\vspace{0.3cm}\displaystyle{\hspace{1cm}\leq \sum_{j \neq i_1 ,L} \alpha_{i_1j}(t)\chi_{i_1j} b_{i_1j}(t)( M_{0}^v-\langle x_{i_1}(t),v\rangle)}\\
			\vspace{0.3cm}\displaystyle{\hspace{2cm}+\alpha_{i_1L}(t)b_{i_1L}(t)\left(M_0^v-e^{-K(\gamma(T+\tau)+\tau)}(M_0^v-\tilde{m}_{0}^v)-\langle x_{i_1}(t),v\rangle\right)}\\
			\vspace{0.3cm}\displaystyle{\hspace{1cm}=( M_{0}^v-\langle x_{i_1}(t),v\rangle)\sum_{j \neq i_1 , L} \chi_{i_1j}\alpha_{i_1j}(t) b_{i_1j}(t)}\\
			\displaystyle{\hspace{2cm}+ \alpha_{i_1L}(t)b_{i_1L}(t)\left(M_0^v-e^{-K(\gamma(T+\tau)+\tau)}(M_0^v-\tilde{m}_{0}^v)-\langle x_{i_1}(t),v\rangle\right).}
		\end{array}$$
		Note that
		$$\sum_{j \neq i_1 ,L} \chi_{i_1j} \alpha_{i_1j}(t)b_{i_1j}(t)=\sum_{\substack{j \neq i_1 }} \chi_{i_1j}\alpha_{i_1j}(t) b_{i_1j}(t)-\alpha_{i_1L}(t)b_{i_1L}(t)$$$$\leq \frac{K}{N-1}\sum_{\substack{j \neq i_1 }} \chi_{i_1j}-\alpha_{i_1L}(t)b_{i_1L}(t)=\frac{KN_{i_1}}{N-1}-\alpha_{i_1L}(t)b_{i_1L}(t).$$
		Thus, it comes that
		$$\begin{array}{l}
			\vspace{0.3cm}\displaystyle{\frac{d}{d t}  \langle x_{i_1}(t),v\rangle \leq \frac{KN_{i_{1}}}{N-1}(M_{0}^v-\langle x_{i_1}(t),v\rangle) -\alpha_{i_1L}(t)b_{i_1L}(t)(M_{0}^v-\langle x_{i_1}(t),v\rangle) }\\
			\vspace{0.3cm}\displaystyle{\hspace{3cm}+ \alpha_{i_1L}(t)b_{i_1L}(t)\left(M_0^v-e^{-K(\gamma(T+\tau)+\tau)}(M_0^v-\tilde{m}_{0}^v)-\langle x_{i_1}(t),v\rangle\right)}\\
			\vspace{0.3cm}\displaystyle{\hspace{2cm}=\frac{KN_{i_{1}}}{N-1}(M_{0}^v-\langle x_{i_1}(t),v\rangle) -e^{-K(\gamma(T+\tau)+\tau)}(M_0^v-\tilde{m}_{0}^v)\alpha_{i_1L}(t)b_{i_1L}(t)}\\
			\vspace{0.3cm}\displaystyle{\hspace{2cm}\leq \frac{KN_{i_{1}}}{N-1}(M_{0}^v-\langle x_{i_1}(t),v\rangle) -e^{-K(\gamma(T+\tau)+\tau)}(M_0^v-\tilde{m}_{0}^v)\alpha_{i_1L}(t)\frac{\psi_{0}}{N-1}}\\
			\displaystyle{\hspace{2cm}=\frac{KN_{i_{1}}}{N-1} M_0^v-e^{-K(\gamma(T+\tau)+\tau)}(M_0^v-\tilde{m}_{0}^v)\alpha_{i_1L}(t)\frac{\psi_{0}}{N-1}-\frac{KN_{i_{1}}}{N-1}\langle x_{i_1}(t),v\rangle).}
		\end{array}$$
		Hence, the Gronwall's estimate yields
		$$ \begin{array}{l}
			\vspace{0.3cm}\displaystyle{\langle x_{i_1}(t),v\rangle\leq e^{-\frac{KN_{i_1}}{N-1}(t-\tau)} \langle x_{i_1}(\tau),v\rangle) +M^v_0(1-e^{-\frac{KN_{i_1}}{N-1}(t-\tau)})}\\
			\vspace{0.3cm}\displaystyle{\hspace{2cm}-e^{-K(\gamma(T+\tau)+\tau)}(M_0^v-\tilde{m}_{0}^v)\frac{\psi_0}{N-1}\int_{\tau}^{t}\alpha_{i_1L}(s)e^{-\frac{KN_{i_1}}{N-1}(t-s)}ds}\\
			\vspace{0.3cm}\displaystyle{\hspace{1cm}\leq e^{-\frac{KN_{i_1}}{N-1}(t-\tau)}M_0^v +M^v_0(1-e^{-\frac{KN_{i_1}}{N-1}(t-\tau)})}\\
			\vspace{0.3cm}\displaystyle{\hspace{1cm}-e^{-K(\gamma(T+\tau)+\tau)}(M_0^v-\tilde{m}_{0}^v)e^{-K\gamma(T+\tau)}\frac{\psi_0}{N-1}\int_{\tau}^{t}\alpha_{i_1L}(s)ds}\\
			\vspace{0.3cm}\displaystyle{\hspace{1cm}=M^v_0-e^{-K(2\gamma(T+\tau)+\tau)}(M_0^v-\tilde{m}_{0}^v)\frac{\psi_0}{N-1}\int_{\tau}^{t}\alpha_{i_1L}(s)ds,}
		\end{array}$$
		for all $t\in [\tau,\gamma(T+\tau)+\tau]$. In particular, for $t\in [T+\tau,\gamma(T+\tau)+\tau]$, we find
		\begin{equation}\label{i_1T}
			\langle x_{i_1}(t),v\rangle\leq M^v_0-e^{-K(2\gamma(T+\tau)+\tau)}(M_0^v-\tilde{m}_{0}^v)\frac{\psi_0}{N-1}\tilde{\alpha},
		\end{equation}
		where here we have used the fact that, from \eqref{PE}, 
		$$\int_{\tau}^{t}\alpha_{i_1L}(s)ds\geq \int_{\tau}^{T+\tau}\alpha_{i_1L}(s)ds\geq \tilde{\alpha}.$$
		Let us note that, if $\gamma=1$, estimate \eqref{i_1T} holds for each agent. If $\gamma>1$, let us consider an index $i_2 \in \{1, \dots, N \} \setminus \{i_1\} $ such that $\chi_{i_2i_1}=1$. Then, for a.e. $t \in [T+2\tau,\gamma(T+\tau)+\tau]$,  from \eqref{i_1T} it comes that
		$$\begin{array}{l}
			\vspace{0.3cm}\displaystyle{\frac{d}{d t}  \langle x_{i_2}(t),v\rangle =  \sum_{j \neq i_1, i_2} \chi_{i_2j}\alpha_{i_2j}(t) b_{i_2j}(t)( \langle x_{j}(t-\tau_{i_2j}(t)),v\rangle-\langle x_{i_2}(t),v\rangle) }\\
			\vspace{0.3cm}\displaystyle{\hspace{2cm}+ \alpha_{i_2i_1}(t)b_{i_2i_1}(t)(\langle x_{i_1}(t-\tau_{i_2i_1}(t)),v\rangle-\langle x_{i_2}(t),v\rangle)}\\
			\vspace{0.3cm}\displaystyle{\hspace{1cm}\leq ( M_{0}^v-\langle x_{i_2}(t),v\rangle)\sum_{j \neq i_1, i_1} \chi_{i_2j}\alpha_{i_2j}(t) b_{i_2j}(t)}\\
			\displaystyle{\hspace{2cm}+ \alpha_{i_2i_1}(t)b_{i_2i_1}(t)\left(M^v_0-e^{-K(2\gamma(T+\tau)+\tau)}(M_0^v-\tilde{m}_{0}^v)\frac{\psi_0}{N-1}\tilde{\alpha}-\langle x_{i_2}(t),v\rangle\right).}
		\end{array}$$
		Thus, arguing as above, 
		$$\begin{array}{l}
			\vspace{0.3cm}\displaystyle{\frac{d}{d t}  \langle x_{i_2}(t),v\rangle \leq \frac{KN_{i_{2}}}{N-1}(M_{0}^v-\langle x_{i_2}(t),v\rangle) -\alpha_{i_2i_1}(t)b_{i_2i_1}(t)(M_{0}^v-\langle x_{i_2}(t),v\rangle) }\\
			\vspace{0.3cm}\displaystyle{\hspace{2cm}+ \alpha_{i_2i_1}(t)b_{i_2i_1}(t)\left(M^v_0-e^{-K(2\gamma(T+\tau)+\tau)}(M_0^v-\tilde{m}_{0}^v)\frac{\psi_0}{N-1}\tilde{\alpha}-\langle x_{i_2}(t),v\rangle\right)}\\
			\vspace{0.3cm}\displaystyle{\hspace{1cm}=\frac{KN_{i_{2}}}{N-1}(M_{0}^v-\langle x_{i_1}(t),v\rangle) -\alpha_{i_2i_1}(t)b_{i_2i_1}(t)e^{-K(2\gamma(T+\tau)+\tau)}(M_0^v-\tilde{m}_{0}^v)\frac{\psi_0}{N-1}\tilde{\alpha}}\\
			\vspace{0.3cm}\displaystyle{\hspace{1cm}\leq \frac{KN_{i_{2}}}{N-1}(M_{0}^v-\langle x_{i_1}(t),v\rangle) -\alpha_{i_2i_1}(t)e^{-K(2\gamma(T+\tau)+\tau)}(M_0^v-\tilde{m}_{0}^v)\left(\frac{\psi_0}{N-1}\right)^2\tilde{\alpha}}\\
			\displaystyle{\hspace{1cm}=\frac{KN_{i_{2}}}{N-1}M_{0}^v-\alpha_{i_2i_1}(t)e^{-K(2\gamma(T+\tau)+\tau)}(M_0^v-\tilde{m}_{0}^v)\left(\frac{\psi_0}{N-1}\right)^2\tilde{\alpha}-\frac{KN_{i_{2}}}{N-1}\langle x_{i_2}(t),v\rangle).}
		\end{array}$$
		Again, using Gronwall's estimate it comes that
		$$ \begin{array}{l}
			\vspace{0.3cm}\displaystyle{\langle x_{i_2}(t),v\rangle\leq e^{-\frac{KN_{i_2}}{N-1}(t-T-2\tau)} \langle x_{i_2}(T+2\tau),v\rangle) +M^v_0(1-e^{-\frac{KN_{i_2}}{N-1}(t-T-2\tau)})}\\
			\vspace{0.3cm}\displaystyle{\hspace{2.5cm}-e^{-K(2\gamma(T+\tau)+\tau)}(M_0^v-\tilde{m}_{0}^v)\left(\frac{\psi_0}{N-1}\right)^2\tilde{\alpha}\int_{T+2\tau}^{t}\alpha_{i_2i_1}(s)e^{-\frac{KN_{i_2}}{N-1}(t-s)}ds}\\
			\displaystyle{\hspace{1.7cm}\leq M^v_0-e^{-K(3\gamma(T+\tau)-T)}(M_0^v-\tilde{m}_{0}^v)\left(\frac{\psi_0}{N-1}\right)^2\tilde{\alpha}\int_{T+2\tau}^{t}\alpha_{i_2i_1}(s)ds,}
		\end{array}$$
		for all $t\in [T+2\tau,\gamma (T+\tau)+\tau]$. In particular, for $t\in [2T+2\tau,\gamma (T+\tau)+\tau]$, the condition \eqref{PE} yields	
		\begin{equation}\label{i_2T}
			\langle x_{i_2}(t),v\rangle\leq M^v_0-e^{-K(3\gamma (T+\tau)-T)}(M_0^v-\tilde{m}_{0}^v)\left(\frac{\psi_0}{N-1}\right)^2\tilde{\alpha}^2.
		\end{equation}
		Finally, iterating the above procedure along the path $i_0,i_1,\dots,i_{r},$ $r\le\gamma,$ that starts from $i_0=L$  we find the following upper bound
		\begin{equation} \label{5.13onoff}
			\langle x_{i_k}(t),v\rangle \leq M^v_{0 }-e^{- K((k+1)\gamma(T+\tau) -\left(\sum_{l=0}^{k-1}l\right)(T+\tau)+\tau)} (M_0^v-\tilde{m}_{0}^v)\left(\frac{\psi_{0}\tilde{\alpha}}{N-1}\right)^{k} , 
		\end{equation}
		for all $1\leq k\leq r$ and for all $t\in [k(T+\tau),\gamma (T+\tau)+\tau]$. In particular, if the path has length $\gamma,$ for  $k=\gamma$, since $\sum_{l=0}^{\gamma-1}l=\frac{\gamma(\gamma-1)}{2}$, inequality \eqref{5.13onoff} reads as
		\begin{equation}\label{5.13gammaonoff}
			\langle x_{i_{\gamma}}(t),v\rangle \leq M^v_{0 }-e^{-K (\frac{1}{2}(\gamma^2+3\gamma)(T+\tau)+\tau)}(M_0^v-\tilde{m}_{0}^v) \left(\frac{\psi_{0}\tilde{\alpha}}{N-1}\right)^{\gamma},
		\end{equation}
		for all $t\in [\gamma(T+\tau),\gamma (T+\tau)+\tau]$.
		\\Let us note that \eqref{5.13gammaonoff} holds for every agent in the path starting from $i_0=L$ for $t\in [\gamma(T+\tau),\gamma (T+\tau)+\tau]$. Then, from the arbitrariness of the path and since the
digraph is strongly connected,  \eqref{5.13gammaonoff} holds for all the agents. 
		
Now, let $R=1,\dots,N$ be such that $\tilde{M}_0^v=\langle x_R(\gamma(T+\tau)-\tau),v\rangle$. Then, arguing as before, we get
		\begin{equation}\label{Rinequalityonoff}
			\langle x_R(t),v\rangle	\geq m^v_0(1+e^{-K(\gamma(T+\tau)+\tau)}(\tilde{M}_0^v-m_{0}^v)),\quad \forall t\in [0,\gamma(T+\tau)+\tau].
		\end{equation}
		Employing the same arguments used above, we can conclude that
		$$\langle x_{i}(t),v\rangle \geq m^v_{0 }+e^{-K(\frac{1}{2}(\gamma^2+3\gamma)(T+\tau)+\tau)}(\tilde{M}_0^v-m_{0}^v)\left(\frac{\psi_{0}\tilde{\alpha}}{N-1}\right)^{\gamma} ,$$
		for all $t\in [\gamma(T+\tau),\gamma(T+\tau)+\tau]$ and for all $i=1,\dots,N$. Finally, we can deduce that estimate \eqref{Bonoff} holds.
	\end{proof} 
	The following proposition generalizes the previous one in successive time intervals. Its proof is analogous to the previous one, so we omit it.
	\begin{prop}\label{lemma3'onoff}
		Let $v\in \mathbb{R}^d$. For any $n\in \mathbb{N}_0$, it holds
		\begin{equation} \label{B'onoff}
			m^v_{n}+\Gamma(\tilde{M}^v_{n}-m^v_{n}) \leq \langle x_i(t),v\rangle\leq M^v_{n}-\Gamma(M^v_{n}-\tilde{m}^v_n), 
		\end{equation}
		for all $t\in I_{n+1}$ and for all $i =1,\dots, N$, where $\Gamma$ is the positive constant in \eqref{Gammaonoff}.
	\end{prop}
	Now, we are able the consensus Theorem \ref{consgenonoff}.
	\begin{proof}[\textbf{Proof of Theorem \ref{consgenonoff}}]
		Fix $v\in \RR^d$. Let us define the quantities $$\mathcal{D}^v_n:=M^v_n-m^v_n,\quad \forall n\in \mathbb{N}_0,$$
		where $M_n^v$, $m_n^v$ are the constants introduced in Definition \ref{quantonoff}. Note that, for all $n\in \mathbb{N}_0$, we have $\mathcal{D}^v_n\geq 0$, being $M_n^v\geq m_n^v$.
		\\Let $\Gamma\in (0,1)$ be the constant in \eqref{Gammaonoff}. We claim that
		\begin{equation}\label{D_nonoff}
			\mathcal{D}^v_{n+1}\leq (1-\Gamma)\mathcal{D}_{n}^v,\quad \forall n\in \mathbb{N}_0.
		\end{equation}
		Indeed, fix $n\in \mathbb{N}_0$. Let $i,j=1,\dots,N$ and $s,t\in I_{n+1}$ be such that $\langle x_i(s),v\rangle=M^v_{n+1}$ and $\langle x_j(t),v\rangle=m^v_{n+1}$. Then, applying Lemma \ref{lemma3'onoff}, we can write 
		\begin{equation}\label{calD}
			\begin{array}{l}
				\vspace{0.3cm}\displaystyle{\hspace{1.5cm}\mathcal{D}^v_{n+1}=M^v_{n+1}-m^v_{n+1}=\langle x_i(s),v\rangle-\langle x_j(t),v\rangle}\\
				\displaystyle{\hspace{2cm}\leq M_{n}^v-m_{n}^v-\Gamma(M_{n}^v-\tilde{m}_n^v)-\Gamma(\tilde{M}_n^v-m_{n}^v).}
			\end{array}
		\end{equation}
			Now, we distinguish four cases.
		\\{\bf Case I)} Assume that $M_{n}^v-\tilde{m}_n^v=0$ and  $\tilde{M}_n^v-m_{n}^v=0$. Then, since from \eqref{scalprn}
		$$m_{n}^v\leq \tilde{m}_n^v\leq \tilde{M}_n^v=m_{n}^v,$$
		we get $$m_{n}^v=\tilde{m}_n^v=M_{n}^v.$$
		As a consequence, \eqref{calD} becomes 
		$$\mathcal{D}^v_{n+1}\leq 0=(1-\Gamma) \mathcal{D}^v_{n}.$$
		{\bf Case II)} Assume that $M_{n}^v-\tilde{m}_n^v=0$ and  $\tilde{M}_n^v-m_{n}^v>0$. Then, since from \eqref{scalprn} 
		$$\tilde{m}_n^v\leq \tilde{M}_{n}^v\leq M_{n}^v= \tilde{m}_{n}^v,$$
		we can write $$\tilde{M}_n^v=M_{n}^v.$$
		As a consequence, \eqref{calD} becomes 
		$$\mathcal{D}^v_{n+1}\leq M_{n}^v-m_{n}^v-\Gamma\tilde{M}_n^v+\Gamma m_{n}^v=(1-\Gamma )(M_{n}^v-m_{n}^v)=(1-\Gamma )\mathcal{D}^v_{n}.$$
		{\bf Case III)} Assume that $M_{n}^v-\tilde{m}_n^v>0$ and  $\tilde{M}_n^v-m_{n}^v=0$. Then, from \eqref{scalprn} we have 
		$$m_{n}^v\leq\tilde{m}_{n}^v\leq \tilde{M}_{n}^v=m_{n}^v ,$$
		from which $$\tilde{m}_n^v=m_{n}^v.$$
		As a consequence, \eqref{calD} becomes 
		$$\mathcal{D}^v_{n+1}\leq M_{n}^v-m_{n}^v-\Gamma M_{n}^v+\Gamma \tilde{m}_{n}^v=(1-\Gamma )(M_{n}^v-m_{n}^v)=(1-\Gamma )\mathcal{D}^v_{n}.$$
		{\bf Case IV)} Assume that $M_{n}^v-\tilde{m}_n^v>0$ and  $\tilde{M}_n^v-m_{n}^v>0$. In this case, using the fact that $\tilde{M}_n^v\geq \tilde{m}_n^v$, from \eqref{calD} we get 
		$$\mathcal{D}^v_{n+1}\leq (1-\Gamma) (M_{n}^v-m_{n}^v)-\Gamma \tilde{M}_n^v+\Gamma \tilde{m}_n^v\leq (1-\Gamma) (M_{n}^v-m_{n}^v)=(1-\Gamma)\mathcal{D}^v_{n}.$$
		Hence, \eqref{D_nonoff} is fulfilled.
		\\As a consequence, since the positive constant $\Gamma$ in \eqref{D_nonoff} does not depend of the choice of the vector $v$, we find the following estimate :
		\begin{equation}\label{deconoff}
			D_{n+1}\leq (1-\Gamma)D_{n},\quad \forall n\in \mathbb{N}_0.
		\end{equation}
	To see this, fix $n\in \mathbb{N}$. Let $i,j=1,\dots,N$ and $s,t\in I_{n+1}$ be such that
		$$D_{n+1}=\lvert x_i(s)-x_j(t)\rvert.$$
		Let us define the unit vector $$v=\frac{x_i(s)-x_j(t)}{\lvert x_i(s)-x_j(t)\rvert}.$$
		Then, using \eqref{scalprn} and \eqref{D_nonoff}, 
		$$\begin{array}{l}
			\vspace{0.3cm}\displaystyle{D_{n+1}=\langle x_i(s)-x_j(t),v\rangle=\langle x_i(s),v\rangle-\langle x_j(t),v\rangle}\\
			\vspace{0.3cm}\displaystyle{\hspace{1cm}\leq M^v_{n+1}-m^v_{n+1}=\mathcal{D}^v_{n+1}}\\
			\vspace{0.3cm}\displaystyle{\hspace{1cm}\leq (1-\Gamma)\mathcal{D}^v_{n}=(1-\Gamma)(M^v_{n}-m^v_{n})}\\
			\displaystyle{\hspace{1cm}\leq (1-\Gamma)\max_{k,l=1,\dots,N}\max_{r,w\in I_n}\lvert x_{k}(r)-x_{l}(w)\rvert=(1-\Gamma)D_{n}.}
		\end{array}$$
		Thus, \eqref{deconoff} holds true. 
		\\Now, from \eqref{deconoff} it comes that
		\begin{equation}\label{decayonoff}
			D_{n}\leq (1-\Gamma)^nD_0,\quad \forall n\in \mathbb{N}_0.
		\end{equation}
	 Let us note that \eqref{decayonoff} can be rewritten as
		\begin{equation}\label{decayonoffriscritta}
			D_{n}\leq e^{-nC(\gamma (T+\tau) +\tau)}D_0,\quad \forall n\in \mathbb{N}_0,
		\end{equation}
		where $$C=\frac{1}{\gamma (T+\tau)+\tau}\ln\left(\frac{1}{1-\Gamma}\right).$$
		\\Now, let $t\geq0$. Thus, $t\in [n(\gamma (T+\tau)+\tau),(n+1)(\gamma (T+\tau)+\tau)]$, for some $n\in \mathbb{N}_0$. Then, using \eqref{diamonoff2} and \eqref{decayonoffriscritta}, it comes that
		$$d(t)\leq D_{n}\leq e^{-nC(\gamma (T+\tau)+\tau) }D_0\leq e^{- C(t-\gamma(T+\tau)-\tau)}D_0,$$
		which concludes our proof.
	\end{proof}

\begin{oss}
		Let us note that, in the case $\chi_{ij}=1$, for all $i,j=1,\dots,N$, $i\neq j$, and $\alpha_{ij}(t)=\alpha(t),$ $\tau_{ij}(t)=\tau(t),$ for all $i,j=1,\dots,N$, for  suitable functions $\alpha(\cdot)$ and $\tau(\cdot)$ satisfying \eqref{PE} and \eqref{taubounded} respectively, then the constant $C$ in Theorem \ref{consgenonoff} can be chosen independent of the number of agents (see \cite{ContPi}). Moreover, the exponential convergence to consensus in \cite{ContPi} still holds if, instead of a weight function $\alpha(\cdot)$ and a time delay function $\tau(\cdot)$, we take weight functions and delay functions depending of the pair $i,j,$ under the   following symmetry assumption: $\alpha_{ij}(t)=\alpha_{ji}(t)$, for a.e. $t\geq 0$ and for all $i,j=1,\dots,N$, and $\tau_{ij}(t)=\tau_{ji}(t)$, for all $t\geq 0$ and for all $i,j=1,\dots,N$. So, also in the case of symmetric weight functions and universal interaction, the constant $C$ can be taken independent of the number of agents. However, our result here works in a more general setting than the one considered in \cite{ContPi}.
	\end{oss}
\section{The second-order alignment model}\label{secdelay}
	Consider a finite set of $N\in\N$ particles, with $N\geq 2 $. Let $x_{i}(t)\in \RR^d$ and $v_{i}(t)\in \RR^d$ denote the position and the velocity of the $i$-th particle at time $t$, respectively. We shall denote with $\lvert\cdot \rvert$ and $\langle \cdot,\cdot \rangle$ the usual norm and scalar product in $\RR^{d}$, respectively. The interactions between the elements of the system are described by the following Cucker-Smale type model with a variable time delay
	\begin{equation}\label{csp}
		\begin{cases}
			\frac{d}{dt}x_{i}(t)=v_{i}(t),\quad &t>0, \,\,\forall i=1,\dots,N,\\\frac{d}{dt}v_{i}(t)=\underset{j:j\neq i}{\sum}\chi_{ij}\alpha_{ij}(t)c_{ij}(t)(v_{j}(t-\tau_{ij}(t))-v_{i}(t)),\quad 	&t>0,\,\,\forall i=1,\dots,N,	
		\end{cases}
	\end{equation}
	where the time delay functions $\tau_{ij}:[0,+\infty)\rightarrow[0,+\infty)$ are as in \eqref{taubounded}, the terms $\chi_{ij}$ are defined as in \eqref{chiij} and the weight functions $\alpha_{ij}:[0,+\infty)\rightarrow [0,1]$ are $\mathcal{L}^1$-measurable and satisfy the Persistence Excitation Condition {\bf{(PE)}}.
	\\Here, the communication rates $c_{ij}$ of the form
	\begin{equation}\label{weightcs}
		c_{ij}(t):=\frac{1}{N-1}\tilde\psi( \lvert x_{i}(t)-x_{j}(t-\tau_{ij}(t))\rvert), \quad\forall t>0,\, \forall i,j=1,\dots,N,
	\end{equation}
	where $\tilde\psi:\RR\rightarrow \RR$ is a positive, bounded and continuous function and we denote by
	$$\tilde K:=\lVert \tilde \psi\rVert_{\infty}.$$ 
	The initial conditions
	\begin{equation}\label{incondcs}
		x_{i}(s)=x^{0}_{i}(s),\quad v_{i}(s)=v^{0}_{i}(s), \quad \forall s\in [-\tau,0],\,\forall i=1,\dots,N,
	\end{equation}
	are assumed to be continuous functions.
	\\We set
	\begin{equation}\label{C0}
		C^V_{0}:=\max_{i=1,\dots,N}\,\,\max_{s\in [-\tau, 0]}\lvert v_{i}(s)\rvert,
	\end{equation} 
	\begin{equation}\label{MX}
		M_0^{X}:=\max_{i=1,\dots,N}\,\,\max_{s,t\in [-\tau,0]}\lvert x_{i}(s)-x_i(t)\rvert.
	\end{equation}
	We define the space and velocity diameters as follows $$d_{X}(t):=\max_{i,j=1,\dots,N}\lvert x_{i}(t)-x_{j}(t)\rvert,\quad \forall t\geq -\tau$$
	$$d_{V}(t):=\max_{i,j=1,\dots,N}\lvert v_{i}(t)-v_{j}(t)\rvert,\quad \forall t\geq -\tau.$$
	\begin{defn} \label{unflock} We say that a solution $\{(x_{i},v_{i})\}_{i=1,\dots,N}$ to system \eqref{csp} exhibits \textit{asymptotic flocking} if it satisfies the two following conditions:
		\begin{enumerate}
			\item there exists a positive constant $d^{*}$ such that$$\sup_{t\geq-{\tau}}d_{X}(t)\leq d^{*};$$
			\item$\underset{t \to\infty}{\lim}d_{V}(t)=0.$
		\end{enumerate}
	\end{defn}
	Our main result is the following.
	\begin{thm} \label{uf}
		Assume \eqref{taubounded} and that the digraph $\mathcal{G}$ is strongly connected. Let $\tilde\psi:\RR\rightarrow\RR$ be a positive, bounded, continuous function that satisfies
		\begin{equation}\label{infint}
			\int_{0}^{+\infty}\left(\min_{r\in [0,t]}\tilde\psi(r)\right)^{\gamma}dt=+\infty,
		\end{equation}
		where $\gamma$ is the depth of the digraph. Assume that the weight functions $\alpha_{ij}:[0,+\infty)\rightarrow[0,1]$ are
${\mathcal L^1}$-measurable and satisfy {\bf (PE)}. Moreover, let $x^{0}_{i},v_{i}^{0}:[-{\tau},0]\rightarrow \RR^{d}$ be continuous functions, for any $i=1,\dots,N$. 
		Then, for every solution $\{(x_{i},v_{i})\}_{i=1,\dots,N}$ to \eqref{csp} with the initial conditions \eqref{incondcs}, there exists a positive constant $d^{*}$ such that \begin{equation}\label{posbound}
			\sup_{t\geq-{\tau}}d_{X}(t)\leq d^{*},
		\end{equation}
		and there exists a positive constant $\mu$ for which the following exponential decay estimate holds
		\begin{equation}\label{vel}
			d_{V}(t)\leq \left(\max_{i,j=1,\dots,N}\,\,\max_{r,s\in [-\tau,0]}\lvert v_{i}(r)-v_{j}(s)\rvert\right) e^{-\mu(t-\gamma(T+\tau)-\tau)},\quad \forall t\geq 0.
		\end{equation}
	\end{thm}
\begin{oss}
	Let us note that, if the function $\tilde{\psi}$ is nonincreasing and the interaction is universal, i.e. $\gamma=1$, then the condition \eqref{infint} reduces to
	$$\int_{0}^{+\infty}\tilde\psi(t)dt=+\infty,$$
	which is the classical assumption to obtain the unconditional flocking (see e.g. \cite{Cartabia}). Since here we deal with an influence function not necessarily monotonic and the interaction is not universal,  we require the stronger assumption \eqref{infint} (cf. \cite{Cont} for the case of universal interaction).
\end{oss}
	\subsection{Preliminary lemmas}
	Let $\{x_{i},v_{i}\}_{i=1,\dots,N}$ be solution to \eqref{csp} under the initial conditions \eqref{incondcs}. We assume that the hypotheses of Theorem \ref{uf} are satisfied. The following lemmas hold. We omit their proofs since they can be proved using the same arguments employed in Section \ref{timedelay}.
	\begin{defn}\label{quantcs}
		Given a vector $v\in \mathbb{R}^d$, for all $n\in\mathbb{N}_0$ we define 
	
		$${r}_n^v:=\min_{j=1,\dots,N}\min_{s\in I_n}\,\langle v_{j}(s),v\rangle,$$
		$${R}_n^v:=\max_{j=1,\dots,N}\max_{s\in I_n}\,\langle v_{j}(s),v\rangle,$$
where, as in the previous section,
$$I_n=[n(\gamma (T+\tau)+\tau)-\tau,n(\gamma (T+\tau)+\tau)].$$

		Also, we define, for all $n\in\mathbb{N}_0$,
		$$\tilde{r}_n^v:=\min_{j=1,\dots,N}\langle v_{j}(n(\gamma (T+\tau)+\tau)),v\rangle,$$
		$$\tilde{R}_n^v:=\max_{j=1,\dots,N}\langle v_{j}(n(\gamma (T+\tau)+\tau)),v\rangle.$$
	\end{defn}
	\begin{lem}\label{L1cs}
		For each vector $v\in \RR^{d}$ and for any $n\in \mathbb{N}_0$, we have that 
		\begin{equation}\label{scalprcs}
			r_n^v\leq \langle v_{i}(t),v\rangle \leq R_n^v,
		\end{equation}for all  $t\geq n(\gamma (T+\tau)+\tau)-\tau$ and for any $i=1,\dots,N$.
	\end{lem}

	\begin{defn}\label{Dncs}
		For all $n\in \mathbb{N}_0$, we define 
		$$F_n:=\max_{i,j =1,\dots, N}\max_{r,s\in I_n}\lvert v_i(r)-v_j(s)\rvert.$$
	\end{defn}
	\begin{oss}
		Let us note that $$F_0:=\max_{i,j =1,\dots, N}\max_{r,s\in I_0}\lvert v_i(r)-v_j(s)\rvert=\max_{i,j =1,\dots, N}\max_{r,s\in [-\tau,0]}\lvert v_i(r)-v_j(s)\rvert.$$
		Then, the exponential decay estimate in \eqref{vel} can be written as
		$$d_V(t)\leq e^{-\mu(t-\gamma(T+\tau)-\tau)}F_0,\quad \forall t\geq 0.$$
	\end{oss}
	\begin{lem}\label{lemmadiamcs}
		For each $n\in \mathbb{N}_0$, we have that 
		\begin{equation}\label{diamcs}
			\lvert v_i(s)-v_j(t)\rvert\leq F_n,
		\end{equation}	
		for all  $s,t\geq n(\gamma (T+\tau)+\tau)-\tau$ and for any $i,j=1,\dots,N$.
	\end{lem}
	\begin{oss}
		Let us note that \eqref{diamcs} yields
		\begin{equation}\label{diamcs2}
			d_V(t)\leq F_n,\quad \forall t\geq n(\gamma (T+\tau)+\tau)-\tau.
		\end{equation}
		Furthermore, from \eqref{diamcs} it follows that 
	\begin{equation}\label{decreasing}
		F_{n+1}\leq F_{n},\quad \forall n\in \mathbb{N}_0.
	\end{equation}
	\end{oss}
	Also, arguing as in Section \ref{timedelay}, we can find a bound on the velocities $\vert v_i(t)\vert,$ which is uniform with respect to $t$ and $i=1,\dots,N.$ 
	\begin{lem}\label{L3cs}
		For every $i=1,\dots,N,$ we have that \begin{equation}\label{boundsolcs}
			\lvert v_{i}(t)\rvert\leq C^V_{0},\quad \forall t\geq-\tau,
		\end{equation}
		where $C^V_{0}$ is the constant defined in \eqref{C0}.	
	\end{lem}
Now, we provide the following result in which an estimate on the position diameters is established. 
	\begin{lem}
		For every $i,j=1,\dots,N$, we get
		\begin{equation}\label{dist}
			\lvert x_{i}(t)-x_{j}(t-\tau_{ij}(t))\rvert\leq \tau C_{0}^{V}+M^{X}_{0}+d_{X}(t), \quad\forall t\geq0,
		\end{equation}
		where $C_{0}^{V}$ and $M^{X}_{0}$ are the positive constants in \eqref{C0} and \eqref{MX}, respectively.
	\end{lem}
	\begin{proof}
		Given $i,j=1,\dots,N$ and $t\geq0$, we have
		\begin{equation}\label{split}
			\begin{split}
				\lvert x_{i}(t)-x_{j}(t-\tau_{ij}(t))\rvert&\leq \lvert x_{i}(t)-x_{j}(t)\rvert+\lvert x_{j}(t)-x_{j}(t-\tau_{ij}(t))\rvert\\&\leq d_{X}(t)+\lvert x_{j}(t)-x_{j}(t-\tau_{ij}(t))\rvert.
			\end{split}
		\end{equation}
		Now, we estimate $$\lvert x_{j}(t)-x_{j}(t-\tau_{ij}(t))\rvert.$$
		If $t-\tau_{ij}(t)>0$, from \eqref{taubounded} and \eqref{boundsolcs} we get
		$$\begin{array}{l}
			\vspace{0.3cm}\displaystyle{\lvert x_{j}(t)-x_{j}(t-\tau_{ij}(t))\rvert\leq \int_{t-\tau_{ij}(t)}^{t}\lvert v_{j}(s)\rvert ds}\\
			\displaystyle{\hspace{2cm}\leq C^{V}_{0}\tau_{ij}(t)\leq \tau C^{V}_{0}.}
		\end{array}$$
	On the other hand, if $t-\tau_{ij}(t)\leq 0$, then $t\leq \tau$ and 
	$$\begin{array}{l}
		\vspace{0.3cm}\displaystyle{\lvert x_{j}(t)-x_{j}(t-\tau_{ij}(t))\rvert\leq\lvert x_j(0)-x_j(t-\tau_{ij}(t))\rvert+ \int_{0}^{t}\lvert v_{j}(s)\rvert ds}\\
		\displaystyle{\hspace{2cm}\leq M^{X}_{0}+t C^{V}_{0}\leq M^{X}_0+\tau C^{V}_{0}.}
	\end{array}$$
		Therefore, in both cases,
		$$\lvert x_{j}(t)-x_{j}(t-\tau_{ij}(t))\rvert\leq M^{X}_0+\tau C^{V}_{0},$$
		from which \eqref{split} becomes
		$$	\lvert x_{i}(t)-x_{j}(t-\tau_{ij}(t))\rvert\leq M^{X}_0+\tau C^{V}_{0}+d_{X}(t).$$
	\end{proof}
	\subsection{The flocking estimate }
To prove the flocking result we need, as before, a crucial proposition. 
	First of all, we give the following definition.
	\begin{defn}
		We define
		$$\tilde\phi(t):=\min\left\{\tilde\psi(r):r\in \left[0,\tau C^{V}_{0}+M_{X}^{0}+\max_{s\in[-\tau,t] }d_{X}(s)\right]\right\},$$
		for all $t\geq -\tau$.
	\end{defn}
	\begin{oss}
		Let us note that from \eqref{dist}
		$$\tilde\psi(\lvert x_i(t)-x_j(t-\tau_{ij}(t))\rvert)\geq \tilde\phi(t),\quad\forall t\geq 0,\,\forall i,j=1,\dots,N.$$
		from which
		\begin{equation}\label{weightlowerbound}
			c_{ij}(t)\geq \frac{1}{N-1}\tilde\phi(t),\quad \forall t\geq 0,\,\forall i,j=1,\dots,N.
		\end{equation}
	\end{oss}
	
	\begin{prop}\label{lemma 3cs}
		For all $v\in \mathbb{R}^d$, it holds
		\begin{equation} \label{Bcs}
			r_{0}^v+\Gamma_1(\tilde{R}^v_{0}-r^v_{0}) \leq \langle v_i(t),v\rangle\leq R^v_{0}-\Gamma_1(R^v_{0}-\tilde{r}_{0}^v), 
		\end{equation}
		for all $t\in [\gamma(T+\tau),\gamma(T+\tau)+\tau]$ and for all $i =1,\dots, N$, where $\Gamma_1$ is the positive constant defined as follows
		\begin{equation}\label{Gammacs}
			\Gamma_1:=e^{-K (\frac{1}{2}(\gamma^2+3\gamma)(T+\tau)+\tau)}\left(\frac{\tilde\phi(\gamma(T+\tau)+\tau)\tilde{\alpha}}{N-1}\right)^\gamma.
		\end{equation}
	\end{prop}
	\begin{oss}
		Let us note that, from {\bf (PE)}, $\Gamma_1\in (0,1)$ since $\tilde{\alpha} K\leq 1$.
	\end{oss}
	\begin{proof}
		Fix $v \in \mathbb{R}^d$. Let $L=1,\dots,N$ be such that $\langle v_L(0),v\rangle= \tilde{r}_{0}^v $. Note that from \eqref{scalprcs}, $R^v_0\geq \tilde{r}^v_0$. Then, for a.e. $t\in [0,\gamma(T+\tau)+\tau]$, from \eqref{scalprcs}
		$$\begin{array}{l}
			\vspace{0.3cm}\displaystyle{\frac{d}{dt}\langle v_L(t),v\rangle=\sum_{j:j\neq L}\chi_{Lj}\alpha_{Lj}(t)c_{Lj}(t)(\langle v_j(t-\tau_{Lj}(t)),v\rangle-\langle v_L(t),v\rangle)}\\
			\vspace{0.3cm}\displaystyle{\hspace{1.5cm}\leq \sum_{j:j\neq L}\chi_{Lj}\alpha_{Lj}(t)c_{Lj}(t)(R_0^v-\langle v_L(t),v\rangle)}\\
			\displaystyle{\hspace{1.5cm}\leq \frac{\tilde K}{N-1}\sum_{j:j\neq L}(R_0^v-\langle v_L(t),v\rangle)=\tilde K(R_0^v-\langle v_L(t),v\rangle).}
		\end{array}$$
		Thus, the Gronwall's inequality yields
		$$\begin{array}{l}
			\vspace{0.3cm}\displaystyle{\langle v_L(t),v\rangle\leq e^{-\tilde Kt}\langle v_L(0),v\rangle+R_0^v(1-e^{-\tilde Kt})}\\
			\vspace{0.3cm}\displaystyle{\hspace{1.5cm}= R_0^v-e^{-\tilde Kt}(R_0^v-\tilde{r}_{0}^v)}\\
			\displaystyle{\hspace{1.5cm}\leq R_0^v-e^{-\tilde K(\gamma(T+\tau)+\tau)}(R_0^v-\tilde{r}_{0}^v).}
		\end{array}$$
	Therefore, we have
		\begin{equation}\label{firstgroncs}
			\langle v_L(t),v\rangle\leq R_0^v-e^{-\tilde K(\gamma(T+\tau)+\tau)}(R_0^v-\tilde{r}_{0}^v),\quad \forall t\in [0,\gamma (T+\tau)+\tau].
		\end{equation}
		Now, let $i_1=1,\dots,N\setminus \{L\}$ be such that $\chi_{i_1L}=1$. Such an index $i_1$ exists since the digraph is strongly connected. Then, for a.e. $t \in [\tau,\gamma (T+\tau)+\tau]$, from \eqref{firstgroncs} we get
		$$\begin{array}{l}
			\vspace{0.3cm}\displaystyle{\frac{d}{d t}  \langle v_{i_1}(t),v\rangle =  \sum_{j \neq i_1 , L} \chi_{i_1j}\alpha_{i_1j}(t) c_{i_1j}(t)( \langle v_{j}(t-\tau_{i_1j}(t)),v\rangle-\langle v_{i_1}(t),v\rangle) }\\
			\vspace{0.3cm}\displaystyle{\hspace{2cm}+ \alpha_{i_1L}(t)c_{i_1L}(t)(\langle v_{L}(t-\tau_{i_1L}(t)),v\rangle-\langle v_{i_1}(t),v\rangle)}\\
			\vspace{0.3cm}\displaystyle{\hspace{1cm}\leq \sum_{j \neq i_1, L} \alpha_{i_1j}(t)\chi_{i_1j} c_{i_1j}(t)( R_{0}^v-\langle v_{i_1}(t),v\rangle)}\\
			\vspace{0.3cm}\displaystyle{\hspace{2cm}+\alpha_{i_1L}(t)c_{i_1L}(t)\left(R_0^v-e^{-\tilde K(\gamma(T+\tau)+\tau)}(R_0^v-\tilde{r}_{0}^v)-\langle v_{i_1}(t),v\rangle\right)}\\
			\vspace{0.3cm}\displaystyle{\hspace{1cm}=( R_{0}^v-\langle v_{i_1}(t),v\rangle)\sum_{j \neq i_1 ,L} \chi_{i_1j}\alpha_{i_1j}(t) c_{i_1j}(t)}\\
			\displaystyle{\hspace{2cm}+ \alpha_{i_1L}(t)c_{i_1L}(t)\left(R_0^v-e^{-\tilde K(\gamma(T+\tau)+\tau)}(R_0^v-\tilde{r}_{0}^v)-\langle v_{i_1}(t),v\rangle\right).}
		\end{array}$$
		Note that
		$$\sum_{j \neq i_1 , L} \chi_{i_1j} \alpha_{i_1j}(t)c_{i_1j}(t)=\sum_{\substack{j \neq i_1 }} \chi_{i_1j}\alpha_{i_1j}(t) c_{i_1j}(t)-\alpha_{i_1L}(t)c_{i_1L}(t)$$$$\leq \frac{\tilde K}{N-1}\sum_{j \neq i_1 , L}\chi_{i_1j}-\alpha_{i_1L}(t)c_{i_1L}(t)=\frac{\tilde KN_{i_1}}{N-1}-\alpha_{i_1L}(t)c_{i_1L}(t).$$
		Thus, from \eqref{weightlowerbound} it comes that
		$$\begin{array}{l}
			\vspace{0.3cm}\displaystyle{\frac{d}{d t}  \langle v_{i_1}(t),v\rangle \leq \frac{\tilde KN_{i_{1}}}{N-1}(R_{0}^v-\langle v_{i_1}(t),v\rangle) -\alpha_{i_1L}(t)c_{i_1L}(t)(R_{0}^v-\langle v_{i_1}(t),v\rangle) }\\
			\vspace{0.3cm}\displaystyle{\hspace{3cm}+ \alpha_{i_1L}(t)c_{i_1L}(t)\left(R_0^v-e^{-\tilde K(\gamma(T+\tau)+\tau)}(R_0^v-\tilde{r}_{0}^v)-\langle v_{i_1}(t),v\rangle\right)}\\
			\vspace{0.3cm}\displaystyle{\hspace{2cm}\leq \frac{\tilde KN_{i_{1}}}{N-1}(R_{0}^v-\langle v_{i_1}(t),v\rangle) -\alpha_{i_1L}(t)\frac{\tilde\phi(t)}{N-1}e^{-\tilde K(\gamma(T+\tau)+\tau)}(R_0^v-\tilde{r}_{0}^v)}\\
			\displaystyle{\hspace{2cm}=\frac{\tilde KN_{i_{1}}}{N-1} R_0^v-\alpha_{i_1L}(t)\frac{\tilde \phi(t)}{N-1}e^{-\tilde K(\gamma(T+\tau)+\tau)}(R_0^v-\tilde{r}_{0}^v)-\frac{\tilde  KN_{i_{1}}}{N-1}\langle v_{i_1}(t),v\rangle).}
		\end{array}$$
		Hence, the Gronwall's estimate yields
		$$ \begin{array}{l}
			\vspace{0.3cm}\displaystyle{\langle v_{i_1}(t),v\rangle\leq e^{-\frac{\tilde KN_{i_1}}{N-1}(t-\tau)} \langle v_{i_1}(\tau),v\rangle) +R^v_0(1-e^{-\frac{\tilde KN_{i_1}}{N-1}(t-\tau)})}\\
			\vspace{0.3cm}\displaystyle{\hspace{2cm}-e^{-\tilde K(\gamma(T+\tau)+\tau)}(R_0^v-\tilde{r}_{0}^v)\frac{1}{N-1}\int_{\tau}^{t}\tilde\phi(s)\alpha_{i_1L}(s)e^{-\frac{\tilde KN_{i_1}}{N-1}(t-s)}ds}\\
			\vspace{0.3cm}\displaystyle{\hspace{1cm}\leq e^{-\frac{\tilde KN_{i_1}}{N-1}(t-\tau)}R_0^v +R^v_0(1-e^{-\frac{\tilde KN_{i_1}}{N-1}(t-\tau)})}\\
			\vspace{0.3cm}\displaystyle{\hspace{2cm}-e^{-\tilde K(\gamma(T+\tau)+\tau)}(R_0^v-\tilde{r}_{0}^v)e^{-\tilde K\gamma(T+\tau)}\frac{1}{N-1}\int_{\tau}^{t}\tilde \phi(s)\alpha_{i_1L}(s)ds}\\
			\displaystyle{\hspace{1cm}=R^v_0-e^{-\tilde K(2\gamma(T+\tau)+\tau)}(R_0^v-\tilde{r}_{0}^v)\frac{1}{N-1}\int_{\tau}^{t}\tilde \phi(s)\alpha_{i_1L}(s)ds,}
		\end{array}$$
		for all $t\in [\tau,\gamma(T+\tau)+\tau]$. Note that, since $\tilde\phi$ is a nonincreasing function, 
		\begin{equation}\label{ineqphi}
			\tilde \phi(t)\geq \tilde \phi(\gamma(T+\tau)+\tau),\quad \forall t\in [0,\gamma(T+\tau)+\tau].
		\end{equation}
		Then, we can write
		$$ \begin{array}{l}
			\displaystyle{\langle v_{i_1}(t),v\rangle\leq R^v_0-e^{-\tilde K(2\gamma(T+\tau)+\tau)}(R_0^v-\tilde{r}_{0}^v)\frac{\tilde\phi(\gamma(T+\tau)+\tau)}{N-1}\int_{\tau}^{t}\alpha_{i_1L}(s)ds,}
		\end{array}$$
		for all $t\in [\tau,\gamma(T+\tau)+\tau]$. In particular, for $t\in [T+\tau,\gamma(T+\tau)+\tau]$, we find
		\begin{equation}\label{i_1Tcs}
			\langle v_{i_1}(t),v\rangle\leq R^v_0-e^{-\tilde K(2\gamma(T+\tau)+\tau)}(R_0^v-\tilde{r}_{0}^v)\frac{\tilde\phi(\gamma(T+\tau)+\tau)}{N-1}\tilde{\alpha},
		\end{equation}
		where here we have used the fact that \eqref{PE} implies the following inequality 
		$$\int_{\tau}^{t}\alpha_{i_1L}(s)ds\geq \int_{\tau}^{T+\tau}\alpha_{i_1L}(s)ds\geq \tilde{\alpha}.$$
		Now, if $\gamma=1$, \eqref{i_1Tcs} holds true for each agent. On the other hand, if $\gamma>1$, let us consider an index $i_2 \in \{1, \dots, N \} \setminus \{i_1\} $ such that $\chi_{i_2i_1}=1$. Then, for a.e. $t \in [T+2\tau,\gamma(T+\tau)+\tau]$, from \eqref{i_1Tcs} it comes that
		$$\begin{array}{l}
			\vspace{0.3cm}\displaystyle{\frac{d}{d t}  \langle v_{i_2}(t),v\rangle =  \sum_{j \neq i_1 , i_2} \chi_{i_2j}\alpha_{i_2j}(t) c_{i_2j}(t)( \langle v_{j}(t-\tau_{i_2j}(t)),v\rangle-\langle v_{i_2}(t),v\rangle) }\\
			\vspace{0.3cm}\displaystyle{\hspace{2cm}+ \alpha_{i_2i_1}(t)c_{i_2i_1}(t)(\langle v_{i_1}(t-\tau_{i_2i_1}(t)),v\rangle-\langle v_{i_2}(t),v\rangle)}\\
			\vspace{0.3cm}\displaystyle{\hspace{1cm}\leq ( R_{0}^v-\langle v_{i_2}(t),v\rangle)\sum_{j \neq i_1, i_2} \chi_{i_2j}\alpha_{i_2j}(t) c_{i_2j}(t)}\\
			\displaystyle{\hspace{2cm}+ \alpha_{i_2i_1}(t)c_{i_2i_1}(t)\left(R^v_0-e^{-\tilde K(2\gamma(T+\tau)+\tau)}(R_0^v-\tilde{r}_{0}^v)\frac{\tilde\phi(\gamma(T+\tau)+\tau)}{N-1}\tilde{\alpha}-\langle v_{i_2}(t),v\rangle\right).}
		\end{array}$$
		Hence, arguing as above we obtain 
		$$\begin{array}{l}
			\vspace{0.3cm}\displaystyle{\frac{d}{d t}  \langle v_{i_2}(t),v\rangle \leq \frac{\tilde KN_{i_{2}}}{N-1}(R_{0}^v-\langle x_{i_2}(t),v\rangle) -\alpha_{i_2i_1}(t)c_{i_2i_1}(t)(R_{0}^v-\langle v_{i_2}(t),v\rangle) }\\
			\vspace{0.3cm}\displaystyle{\hspace{2cm}+ \alpha_{i_2i_1}(t)c_{i_2i_1}(t)\left(R^v_0-e^{-\tilde K(2\gamma(T+\tau)+\tau)}(R_0^v-\tilde{r}_{0}^v)\frac{\tilde \phi(\gamma(T+\tau)+\tau)}{N-1}\tilde{\alpha}-\langle x_{i_2}(t),v\rangle\right)}\\
			\displaystyle{\hspace{1cm}\leq \frac{\tilde KN_{i_{2}}}{N-1}(R_{0}^v-\langle v_{i_2}(t),v\rangle) -\alpha_{i_2i_1}(t)e^{-\tilde K(2\gamma(T+\tau)+\tau)}(R_0^v-\tilde{r}_{0}^v)\frac{\tilde \phi(\gamma(T+\tau)+\tau)}{(N-1)^2}\tilde\phi(t)\tilde{\alpha}.}
		\end{array}$$
		Again, using Gronwall's estimate it comes that
		$$ \begin{array}{l}
			\vspace{0.3cm}\displaystyle{\langle v_{i_2}(t),v\rangle\leq e^{-\frac{\tilde KN_{i_2}}{N-1}(t-T-2\tau)} \langle v_{i_2}(T+2\tau),v\rangle) +R^v_0(1-e^{-\frac{\tilde KN_{i_2}}{N-1}(t-T-2\tau)})}\\
			\vspace{0.3cm}\displaystyle{\hspace{1cm}-e^{-\tilde K(2\gamma(T+\tau)+\tau)}(R_0^v-\tilde{r}_{0}^v)\frac{\tilde \phi(\gamma(T+\tau)+\tau)}{(N-1)^2}\tilde{\alpha}\int_{T+2\tau}^{t}\tilde \phi(s)\alpha_{i_2i_1}(s)e^{-\frac{\tilde KN_{i_2}}{N-1}(t-s)}ds}\\
			\displaystyle{\hspace{0.5cm}\leq R^v_0-e^{-\tilde K(3\gamma (T+\tau)-T)}(R_0^v-\tilde{r}_{0}^v)\frac{\tilde \phi(\gamma(T+\tau)+\tau)}{(N-1)^2}\tilde{\alpha}\int_{T+2\tau}^{t}\tilde \phi(s)\alpha_{i_2i_1}(s)ds,}
		\end{array}$$
		for all $t\in [T+2\tau,\gamma (T+\tau)+\tau]$. In particular, for $t\in [2T+2\tau,\gamma (T+\tau)+\tau]$, the condition \eqref{PE} and the inequality \eqref{ineqphi} imply that	
		\begin{equation}\label{i_2Tcs}
			\langle v_{i_2}(t),v\rangle\leq R^v_0-e^{-\tilde K(3\gamma (T+\tau)-T)}(R_0^v-\tilde{r}_{0}^v)\left(\frac{\tilde \phi(\gamma(T+\tau)+\tau)}{N-1}\right)^2\tilde{\alpha}^2.
		\end{equation}
		Finally, iterating the above procedure along the path $i_0,i_1,\dots,i_{r},$ with $r\le\gamma,$ starting from $i_0=L$  we find the following upper bound
		\begin{equation} \label{5.13cs}
			\langle v_{i_k}(t),v\rangle \leq R^v_{0 }-e^{- K((k+1)\gamma(T+\tau) -(T+\tau)\left(\sum_{l=0}^{k-1}l\right)+\tau)} (R_0^v-\tilde{r}_{0}^v)\left(\frac{\tilde\phi(\gamma(T+\tau)+\tau)\tilde{\alpha}}{N-1}\right)^{k} , 
		\end{equation}
		for all $1\leq k\leq r$ and for all $t\in [k(T+\tau),\gamma (T+\tau)+\tau]$. In particular, if the path has length $\gamma$, for $k=\gamma$, since $\sum_{l=0}^{\gamma-1}l=\frac{\gamma(\gamma-1)}{2}$, inequality \eqref{5.13cs} reads as
		\begin{equation}\label{5.13gammacs}
			\langle v_{i_{\gamma}}(t),v\rangle \leq R^v_{0 }-e^{-K (\frac{1}{2}(\gamma^2+3\gamma)(T+\tau)+\tau)}(R_0^v-\tilde{r}_{0}^v) \left(\frac{\tilde\phi(\gamma(T+\tau)+\tau)\tilde{\alpha}}{N-1}\right)^{\gamma},
		\end{equation}
		for all $t\in [\gamma(T+\tau),\gamma (T+\tau)+\tau]$.
		Arguing as in Proposition \ref{lemma 3onoff}, we can say that \eqref{5.13gammacs} holds for every $i=1,\dots,N$.
		\\Now, let $R=1,\dots,N$ be such that $\tilde{R}_0^v=\langle v_R(0),v\rangle$. Then, arguing as before, we get
		\begin{equation}\label{Rinequalitycs}
			\langle v_R(t),v\rangle	\geq r^v_0+e^{-K(\gamma(T+\tau)+\tau)}(\tilde{R}_0^v-r_{0}^v),\quad \forall t\in [0,\gamma(T+\tau)+\tau].
		\end{equation}
		Employing the same arguments used above, we can conclude that
		$$\langle v_{i}(t),v\rangle \geq r^v_{0 }+e^{-K(\frac{1}{2}(\gamma^2+3\gamma)(T+\tau)+\tau)}(\tilde{R}_0^v-r_{0}^v)) \left(\frac{\tilde\phi(\gamma(T+\tau)+\tau)\tilde{\alpha}}{N-1}\right)^{\gamma} ,$$
		for all $t\in [\gamma(T+\tau),\gamma(T+\tau)+\tau]$ and for all $i=1,\dots,N$. Finally, we can deduce that estimate \eqref{Bcs} holds.
	\end{proof} 
	The following proposition extends the previous one in successive time intervals. We omit its proof since it is analogous to the previous one.
	\begin{prop}\label{lemma3'cs}
		Let $v\in \mathbb{R}^d$. For any $n\in \mathbb{N}$, it holds
		\begin{equation} \label{B'cs}
			r^v_{n}+\Gamma_{n+1}(\tilde{R}^v_{n}-r^v_{n}) \leq \langle v_i(t),v\rangle\leq R^v_{n}-\Gamma_{n+1}(R^v_{n}-\tilde{r}^v_n), 
		\end{equation}
		for all $t\in I_{n+1}$ and for all $i =1,\dots, N$, where $\Gamma_{n+1}$ is the positive constant defined as
		\begin{equation}\label{Gammancs}
			\Gamma_{n+1}:=e^{-K (\frac{1}{2}(\gamma^2+3\gamma)(T+\tau)+\tau)}\left(\frac{\tilde\phi((n+1)(\gamma(T+\tau)+\tau))\tilde{\alpha}}{N-1}\right)^\gamma.
		\end{equation}
	\end{prop}
	\begin{oss}
		Let us note that from \eqref{B'cs} it comes that
		\begin{equation}\label{claim1}
			R_{n+1}^v-r_{n+1}^v\leq (1-\Gamma_{n+1})(R_{n}^v-r_{n}^v),\quad\forall n\in \mathbb{N}_0.
		\end{equation}
		where $\Gamma_{n+1}\in (0,1)$ is the constant in \eqref{Gammancs}.
		\\Indeed, given $n\in \mathbb{N}_0$, let $i,j=1,\dots,N$ and $s,t\in I_{n+1}$ be such that $\langle v_i(s),v\rangle=R^v_{n+1}$ and $\langle v_j(t),v\rangle=r^v_{n+1}$. Then, applying Lemma \ref{lemma3'cs}, we can write 
		\begin{equation}\label{calDpcs}
			\begin{array}{l}
				\vspace{0.3cm}\displaystyle{\hspace{1.5cm}R^v_{n+1}-r^v_{n+1}=\langle v_i(s),v\rangle-\langle v_j(t),v\rangle}\\
				\displaystyle{\hspace{3cm}\leq R_{n}^v-r_{n}^v-\Gamma_{n+1}(R_{n}^v-\tilde{r}_n^v)-\Gamma_{n+1}(\tilde{R}_n^v-r_{n}^v).}
			\end{array}
		\end{equation}
		Then, arguing as in the proof of Theorem \ref{consgenonoff}, we get that estimate \eqref{claim1} holds true.
		\\Also, setting $C^*:=e^{-K (\frac{1}{2}(\gamma^2+3\gamma)(T+\tau)+\tau)}\left(\frac{\tilde{\alpha}}{N-1}\right)^\gamma,$
		it holds that
		\begin{equation}\label{gammaneq}
			\Gamma_{n+1}=C^*(\tilde\phi((n+1)(\gamma(T+\tau)+\tau)))^\gamma,\quad \forall n\in\mathbb{N}_0.
		\end{equation}
	As a consequence, \eqref{claim1} can be written as
	\begin{equation}\label{claim1eq}
		R_{n+1}^v-r_{n+1}^v\leq (1-C^*(\tilde\phi((n+1)(\gamma(T+\tau)+\tau)))^\gamma)(R_{n}^v-r_{n}^v),\quad\forall n\in \mathbb{N}_0.
	\end{equation}
In particular, from \eqref{claim1} and \eqref{claim1eq}, arguing as in Theorem \ref{consgenonoff}, it comes that
\begin{equation}\label{F_n}
	F_{n+1}\leq (1-\Gamma_{n+1})F_n,\quad \forall n\in \mathbb{N}_0,
\end{equation}
or, equivalently,
\begin{equation}\label{F_neq}
	F_{n+1}\leq (1-C^*(\tilde\phi((n+1)(\gamma(T+\tau)+\tau)))^\gamma)F_n,\quad \forall n\in \mathbb{N}_0,
\end{equation}
	\end{oss}
Now, we are able to prove Theorem \ref{uf}.
	\begin{proof}[Proof of Theorem \ref{uf}]
		Let $\{(x_{i},v_{i})\}_{i=1,\dots,N}$  be solution to \eqref{csp} under the initial conditions \eqref{incondcs}. Let us define
		$$\tilde{\Gamma}_{n+1}=\frac{\Gamma_{n+1}}{\gamma(T+\tau)+\tau}, \quad\forall n\in \mathbb{N}_0.$$
		Then, for all $n\in\mathbb{N}_0$ it holds $\tilde{\Gamma}_{n+1}<\Gamma_{n+1}$.
		\\Let us introduce the function $\mathcal{E}:[-\tau,+\infty)\rightarrow [0,+\infty),$ 
		$$\mathcal{E}(t):=\begin{cases}
			F_{0}, \hspace{4.5cm}t\in [-\tau,\gamma(T+\tau)+\tau],\\\mathcal{E}(n(\gamma(T+\tau)+\tau))\left(1-\tilde\Gamma_{n+1}(t-n(\gamma(T+\tau)+\tau))\right), \\\vspace{0.2cm}\hspace{5cm}t\in (n(\gamma(T+\tau)+\tau),(n+1)(\gamma(T+\tau)+\tau)],\,n\geq 1.
		\end{cases}$$
		By definition, $\mathcal{E}$ is continuous, positive and nonincreasing. Moreover, we claim that \begin{equation}\label{boundIn}
			F_{n}\leq \mathcal{E}(t),\quad \forall t\in [-\tau,n(\gamma(T+\tau)+\tau)],\,\forall n\in\mathbb{N}_0.
		\end{equation} 
We prove this by induction. For $n= 1$, from \eqref{decreasing} we can immediately say that $$F_{1}\leq F_{0}=\mathcal{E}(t),\quad \forall  t\in [-\tau,\gamma(T+\tau)+\tau].$$
		Now, assume that \eqref{boundIn} holds for some $n\geq 1$. We have to show that \eqref{boundIn} is true also for $n+1$. From the induction hypothesis and by using again \eqref{decreasing}, we have that $$F_{n+1}\leq F_{n}\leq \mathcal{E}(t),$$
		for all $t\in [-\tau,n(\gamma(T+\tau)+\tau)]$. It lasts to prove that $F_{n+1}\leq \mathcal{E}(t)$, for all $t\in (n(\gamma(T+\tau)+\tau), (n+1)(\gamma(T+\tau)+\tau)]$. From \eqref{F_n}, it comes that
		$$\mathcal{E}(t)\geq \mathcal{E}((n+1)(\gamma(T+\tau)+\tau))= \mathcal{E}(n(\gamma(T+\tau)+\tau))(1-\tilde\Gamma_{n+1}(\gamma(T+\tau)+\tau))= (1-\Gamma_{n+1})F_n\geq F_{n+1},$$
	for all $t\in (n(\gamma(T+\tau)+\tau), (n+1)(\gamma(T+\tau)+\tau)]$, where in the above inequalities we have used the fact that $\mathcal{E}$ is nonincreasing. Hence, \eqref{boundIn} is proven.
	\\Now, for almost all time (see \cite{Cont} for further details)
		\begin{equation}\label{derdiamX}
			\frac{d}{dt}\max_{s\in [-\bar{\tau},t]}d_{X}(s)\leq \left\lvert \frac{d}{dt}d_{X}(t)\right\rvert\leq  d_{V}(t).
		\end{equation}
		Next, let us define the function $\mathcal{W}:[-\tau,+\infty)\rightarrow [0,+\infty)$, 
$$
\mathcal{W}(t):=(\gamma(T+\tau)+\tau)\mathcal{E}(t)+C^*\int_{0}^{\tau C^{V}_{0}+M^{X}_{0}+\underset{s\in [-\tau,t+\gamma(T+\tau)+\tau]}{\max}d_{X}(s)}\left(\min_{\sigma\in [0,r]}\tilde\psi(\sigma)\right)^{\gamma}\,dr,
$$
		for all $t\geq -\tau$. By construction, $\mathcal{W}$ is continuous. Also, for each $n\geq 1$ and for a.e. $t\in(n(\gamma(T+\tau)+\tau),(n+1)(\gamma(T+\tau)+\tau) $, from \eqref{diamcs2}, \eqref{boundIn} and \eqref{derdiamX} it follows that $$\begin{array}{l}
			\vspace{0.3cm}\displaystyle{\frac{d}{dt}\mathcal{W}(t)=(\gamma(T+\tau)+\tau)\frac{d}{dt}\mathcal{E}(t)+C^*(\tilde\phi(t+\gamma(T+\tau)+\tau))^{\gamma}\,\frac{d}{dt}\underset{s\in [-\tau,t+(\gamma(T+\tau)+\tau)]}{\max}d_{X}(s)}\\
			\vspace{0.3cm}\displaystyle{\hspace{0.5cm}\leq-\mathcal{E}(n\gamma(T+\tau)+\tau)C^*(\tilde{\phi}((n+1)(\gamma(T+\tau)+\tau)))^\gamma}\\
\vspace{0.3cm}\hspace{3,5 cm} \displaystyle{
+C^*(\tilde{\phi}(t+\gamma(T+\tau)+\tau))^{\gamma}d_V(t+(\gamma(T+\tau)+\tau))}\\
			\displaystyle{\hspace{0.5cm}\leq C^*F_n(-(\tilde{\phi}((n+1)(\gamma(T+\tau)+\tau)))^{\gamma} +(\tilde{\phi}((n+1)(\gamma(T+\tau)+\tau))^{\gamma}))=0.}
		\end{array}$$
		Then, \begin{equation}\label{negder}
			\frac{d}{dt}\mathcal{W}(t)\leq0,\quad \text{a.e. }t>\gamma(T+\tau)+\tau,
		\end{equation}
		which implies \begin{equation}\label{2tau1}
			\mathcal{W}(t)\leq \mathcal{W}(\gamma(T+\tau)+\tau),\quad \forall t\geq \gamma(T+\tau)+\tau.
		\end{equation}
		Now, by definition of $\mathcal{W}$, being $\mathcal{E}$ a nonnegative function, we have
		$$C^*\int_{0}^{\tau C^{V}_{0}+M^{X}_{0}+\underset{s\in [-\tau,t+\gamma(T+\tau)+\tau]}{\max}d_{X}(s)}\left(\min_{\sigma\in [0,r]}\tilde\psi(\sigma)\right)^{\gamma}\,dr\leq\mathcal{W}(\gamma(T+\tau)+\tau),
		$$
		for all $t\geq\gamma(T+\tau)+\tau$. Letting $t\to \infty$ in the above inequality, we can conclude that \begin{equation}\label{lim2}
			C^*\int_{0}^{\tau C^{V}_{0}+M^{X}_{0}+\underset{s\in [-\tau,+\infty)}{\sup}d_{X}(s)}\left(\min_{\sigma\in [0,r]}\tilde\psi(\sigma)\right)^{\gamma}\,dr\leq\mathcal{W}(\gamma(T+\tau)+\tau)).
		\end{equation} 
		Finally, since the function $\tilde\psi$ satisfies property \eqref{infint}, from \eqref{lim2}, we can conclude that there exists a positive constant $d^{*}$ such that \begin{equation}\label{firstcond}
			\tau C^{V}_{0}+M^{X}_{0}+\underset{s\in [-\tau,+\infty)}{\sup}d_{X}(s)\leq d^{*}.
		\end{equation}
		Now, let us define
		$$\hat\phi:=\min_{r\in[0,d^{*}]}\tilde\psi(r).$$
		Note that $\tilde\phi^{*}>0$. Also, \eqref{firstcond} yields 
		\begin{equation}\label{lowerboundphi}
			\hat\phi\leq \tilde\phi(t), \quad \forall t\geq -\tau.
		\end{equation}
		Then, from \eqref{F_neq} 
		and  \eqref{lowerboundphi} we have 
		\begin{equation}\label{decunif}
			F_{n+1}\leq (1-C^* \hat\phi^\gamma)F_{n},\quad\forall n\in\mathbb{N}_0.
		\end{equation}
		Thus, thanks to an induction argument, we can write
		$$F_{n}\leq (1-C^* \hat\phi^\gamma)^nF_0,\quad \forall n\in \mathbb{N}_0.$$
		Note that the above inequality can be rewritten as
		\begin{equation}\label{finalmente}
			F_{n}\leq e^{-n\mu(\gamma(T+\tau)+\tau)}F_0,\quad \forall n\in\mathbb{N}_0,
		\end{equation}
		where $$\mu=\frac{1}{\gamma(T+\tau)+\tau}\ln\left(\frac{1}{1- C^* \hat\phi^\gamma}\right).$$
		Finally, let $t\geq 0$. Then, $t\in [n(\gamma(T+\tau)+\tau),(n+1)(\gamma(T+\tau)+\tau)]$, for some $n\in \mathbb{N}_0$. Then, using \eqref{diamcs2} and \eqref{finalmente}
		$$d_{V}(t)\leq F_{n}\leq e^{-n\mu(\gamma(T+\tau)+\tau)}F_0 \leq e^{-\mu(t-\gamma(T+\tau)-\tau)}F_0,$$
		which concludes our proof.
	\end{proof}

	\bigskip
	\noindent {\bf Acknowledgements.} {\small The authors are members of Gruppo Nazionale per l’Analisi Matematica,
la Probabilità e le loro Applicazioni (GNAMPA) of the Istituto Nazionale di
Alta Matematica (INdAM). E. Continelli and C. Pi\-gnot\-ti are partially
supported by PRIN 2022  (2022238YY5) {\it Optimal control problems: analysis,
approximation and applications}, and by INdAM GNAMPA Project {\it Modelli alle derivate parziali per interazioni multiagente non 
simmetriche}(CUP E53C23001670001).  C. Pignotti is also partially supported by
PRIN-PNRR 2022 (P20225SP98) {\it Some mathematical approaches to climate change and its impacts}.}

\end{document}